\documentclass[a4paper]{amsart}
\usepackage{graphicx} 
\usepackage{amsmath}
\usepackage{amsfonts}
\usepackage{mathrsfs}
\usepackage{graphicx}
\usepackage{amsmath,amsthm,amssymb,amscd}
\usepackage{color}
\usepackage{xcolor}
\colorlet{MAGENTA}{magenta}
\usepackage[shortlabels]{enumitem}
\usepackage[T1]{fontenc}
\usepackage{tikz-cd}
\usepackage{tikz}
\usepackage{hyperref}
\usetikzlibrary{decorations.pathreplacing,
                calligraphy}

\newtheorem{thm}{Theorem}[section]
\newtheorem{lem}[thm]{Lemma}

\newtheorem{defn}[thm] {Definition}
\newtheorem{exmp}[thm]{Example}

\newtheorem{rem} [thm]{Remark}

\newtheorem{pro}[thm]{Problem}

\newtheorem{theoremA}{Theorem}

\newtheorem{theoremB}{Theorem}

\DeclareMathOperator{\diam}{diam}
\DeclareMathOperator{\dist}{dist}
\DeclareMathOperator{\Per}{Per}
\DeclareMathOperator{\Int}{Int}
\DeclareMathOperator{\Rec}{Rec}

\DeclareMathOperator{\Fix}{Fix}
\DeclareMathOperator{\Crit}{Crit}
\DeclareMathOperator{\Orb}{Orb}

\DeclareMathOperator{\CR}{CR}

\newcommand{\hf}{\hat{f}}

\newcommand{\Astat}{A_{\text{stat}}}
\newcommand{\Amin}{A_{\text{phys}}}

\def\l{{\cal L}}

\newcommand{\N}{{\mathbb{N}}}

\newcommand{\tf}{\tilde{f}}

\tikzset{
B/.style = {decorate,
            decoration={calligraphic brace, amplitude=5pt,
            raise=7pt, },
            very thick,
            pen colour=black},
            }
\tikzset{
C/.style = {decorate,
            decoration={calligraphic brace, amplitude=5pt,
            raise=7pt, mirror},
            very thick,
            pen colour=black},
dot/.style = {circle, fill, inner sep=2pt, outer sep=0pt}
        }

\title[Observable Dynamics and the Generic Coincidence of Attractors]{Observable Dynamics and the Generic Coincidence of Milnor, Statistical, and Physical Attractors}
\date{\today}

	\author[M.\ Fory\'s-Krawiec]{{Magdalena Fory\'s-Krawiec}}
	\author[J. Hant\'akov\'a]{{Jana Hant\' akov\' a}}
	\author[M.\ Kowalewski]{Micha\l{}~Kowalewski}
	\author[P.\ Oprocha]{Piotr Oprocha}

\address[M.\ Fory\'s-Krawiec]{AGH University of Krakow, Faculty of Applied Mathematics,
al.\ Mickiewicza 30, 30-059 Krak\'ow, Poland.}
\email{maforys@agh.edu.pl}
\address[J. Hant\' akov\' a]{Silesian University in Opava, Mathematical Institute, Na Rybn\' i\v cku 626/1, 746 01 Opava, Czech Republic}
\email{jana.hantakova@math.slu.cz}
\address[M. Kowalewski]{AGH University of Krakow, Faculty of Applied Mathematics,
al.\ Mickiewicza 30, 30-059 Krak\'ow, Poland.}
\email{kowalewski@agh.edu.pl}

 \address[P.\ Oprocha]{
Centre of Excellence IT4Innovations - Institute for Research and Applications of Fuzzy Modeling, University of Ostrava, 30. dubna 22, 701 03 Ostrava 1, Czech Republic
-- $\&$ --
AGH University of Krakow, Faculty of Applied Mathematics,
al.\ Mickiewicza 30, 30-059 Krak\'ow, Poland.}
\email{piotr.oprocha@osu.cz}

\newcommand{\Graph}[1]{%
  \ifnum#1=0
    \draw[line width=0.1pt] (0,2/3) -- (1,1/4);
  \else
    \pgfmathtruncatemacro{\nextlevel}{#1-1}
    \draw[line width=0.1pt] (0,2/3) -- (1/3,1);
    \draw[line width=0.1pt] (1/3,1) -- (2/3,2/27);
    \begin{scope}[cm={1/9,0,0,1/9,(2/3,0)}]
      \Graph{\nextlevel}
    \end{scope}
    \draw[line width=0.1pt] (7/9,1/36) -- (8/9,8/27);
    \begin{scope}[cm={1/9,0,0,1/9,(8/9,2/9)}]
      \Graph{\nextlevel}
    \end{scope}
  \fi
}
\newcommand{\InnerGrid}[1]{%
  \ifnum#1>0
    \pgfmathtruncatemacro{\nextlevel}{#1-1}
    \draw[gray!25,densely dotted]
      (1/3,0) -- (1/3,1) (2/3,0) -- (2/3,1)
      (0,1/3) -- (1,1/3) (0,2/3) -- (1,2/3);
    \begin{scope}[cm={1/3,0,0,1/3,(0,2/3)}]\InnerGrid{\nextlevel}\end{scope}
    \begin{scope}[cm={1/3,0,0,1/3,(2/3,0)}]\InnerGrid{\nextlevel}\end{scope}
  \fi
}
\newcommand{\FirstGrid}[1]{%
  \draw[gray!25,densely dotted]
    (1/3,0) -- (1/3,1) (2/3,0) -- (2/3,1)
    (0,1/3) -- (1,1/3) (0,2/3) -- (1,2/3);
  \ifnum#1>1
    \pgfmathtruncatemacro{\nextlevel}{#1-1}
    \begin{scope}[cm={1/3,0,0,1/3,(0,0)}]\InnerGrid{\nextlevel}\end{scope}
    \begin{scope}[cm={1/3,0,0,1/3,(2/3,2/3)}]\InnerGrid{\nextlevel}\end{scope}
  \fi
}

\begin{document}
\begin{abstract}
We study observable dynamics of generic continuous interval maps. We prove that, for a residual subset of $C([0,1])$, the global Milnor, statistical, and physical attractors coincide with the non-wandering set. Thus, for a typical interval map, the asymptotic dynamics detected by Lebesgue-almost every initial condition recovers all recurrent dynamics. The common attractor is a robust Cantor set of zero Hausdorff dimension. Moreover, every invariant probability measure has a basin of zero Lebesgue measure. We also present examples showing that, outside the generic setting, positive topological entropy may occur outside the global Milnor attractor, and the three attractors need not coincide.

\medskip
\noindent \textbf{Keywords:} interval dynamics; Milnor attractor; statistical attractor; physical attractor; generic dynamics.

\noindent \textbf{2020 Mathematics Subject Classification:}
Primary 37E05, 37B35; Secondary 37A05.
\end{abstract}
\maketitle
\section{Introduction}
For a continuous dynamical system, the non-wandering set records all recurrent topological dynamics, whereas Milnor-type attractors describe only the asymptotic behavior visible from a large set of initial conditions. In general, these two objects may differ substantially. The purpose of this paper is to determine their relationship for Baire-generic continuous self-maps of the interval. Our main result reveals a strong connection between these two notions for a typical $f\in C([0,1])$, proving that in these dynamical systems the observable asymptotic dynamics recovers the entire non-wandering set.

There are several natural ways to formalize observable dynamics. Milnor's global attractor~\cite{Milnor} is the smallest closed set containing the $\omega$-limit sets of Lebesgue-almost every point. The statistical attractor of Arnold et al.~\cite{AAIS} ignores sets visited with zero asymptotic frequency, while the physical (or minimal) attractor introduced by Ilyashenko~\cite{Ilyashenko4} is generated by the supports of invariant measures arising from the evolution of Lebesgue measure. These different notions reflect a long-standing discussion about how observability should be formalized mathematically.
Throughout the paper, we use the term ``physical attractor'' only in the above sense, and such an attractor should not be confused with the support of a physical (SRB) measure.

The attractors defined above satisfy the general chain of inclusions
$$
\Amin(f)\subseteq A_{\mathrm{stat}}(f)\subseteq \Lambda(f)\subseteq \Omega(f),
$$
where $\Lambda(f)$ denotes the global Milnor attractor and $\Omega(f)$ the non-wandering set. The first three sets describe, in progressively weaker senses, dynamics observable from Lebesgue-typical initial conditions, whereas the last one contains all recurrent topological dynamics. Thus, the central question addressed in this paper is not merely whether the observable attractors coincide, but whether observable dynamics recovers all recurrent dynamics.

The coincidence problem itself was raised in~\cite{Sharkovsky,Ilyash}:

\begin{pro}[Problem~1, \cite{Ilyash}]\label{pro:Ilyash}
Is it true that, for generic dynamical systems, all the various types of attractors coincide?
\end{pro}

The answer is affirmative in several settings. For example, Okunev~\cite{okunev} proved that, for certain skew products, the Milnor and statistical attractors coincide and are Lyapunov stable. On the other hand, physical and statistical attractors may differ~\cite{Kleptsyn}, and generic systems on compact manifolds (in the topology of uniform convergence) may have no physical measures at all~\cite{AA2013}. Related approaches to observable statistical behavior include pseudo-physical measures~\cite{CatTroub,CE1}.

Our starting point is Mizera's description of generic continuous interval maps~\cite{Miz}. His construction produces a full-measure set of points whose $\omega$-limit sets are solenoids, each with a basin of zero Lebesgue measure. Consequently, Lebesgue-almost every orbit is asymptotically governed by an adding machine. This gives a detailed description of observable dynamics, but does not determine the entire recurrent dynamics. Indeed, an invariant set of zero Lebesgue measure may lie completely outside the observable attractor while still carrying substantial, or even all, of the topological complexity of the system.
In fact, this gap can be maximal: Theorem~\ref{thm:entropy-hidden} exhibits a map whose observable (Milnor) attractor is a zero-entropy union of solenoids, while a disjoint zero-measure Cantor set, at positive distance from it, carries the entire topological entropy of the system. Hence the identity $\Lambda(f)=\Omega(f)$ is in no way automatic.

The main new ingredient of the present paper is a perturbative mechanism that combines Mizera's almost-everywhere description with simultaneous control of the periodic points created throughout the approximation. This allows us to prove
$
\Lambda(f)=\overline{\operatorname{Per}(f)}=\Omega(f).
$
Consequently, for generic interval maps, the observable asymptotic dynamics recovers all recurrent dynamics.

The necessity of our generic assumptions is illustrated by two examples. The first shows that the global Milnor attractor need not contain all recurrent dynamics and may have zero topological entropy while a disjoint zero-measure Cantor repeller carries positive topological entropy. The second shows that the Milnor, statistical, and physical attractors need not coincide.

We denote by $C([0,1])$ the space of continuous self-maps of the interval endowed with the uniform topology.

\begin{theoremA}\label{thm:A}
There exists a residual subset $\mathcal{R} \subset C([0,1])$ such that for every $f \in \mathcal{R}$ 
the non-wandering set $\Omega(f)$ is a Cantor set of zero Hausdorff dimension, and the Milnor attractor, 
the statistical attractor, and the physical attractor of $f$ all coincide and equal $\Omega(f)$.
\end{theoremA}

The key conclusion of Theorem~A is the identity
$
\Lambda(f)=\Omega(f).
$
Coincidence of the observable attractors alone does not imply that they contain all recurrent dynamics. Theorem~A shows that, for generic interval maps, the observable asymptotic dynamics recovers the entire recurrent dynamics of the system.


Robustness means that the map
$
g\longmapsto\Lambda(g)
$
is continuous at $f$ with respect to the Hausdorff metric.

\begin{theoremB}
For every $f$ in the residual family $\mathcal R$, the global Milnor attractor is robust, satisfies
$\Lambda(f)=\Omega(f)$, and is not Lyapunov stable. Moreover, every invariant probability measure of $f$ has a basin of zero Lebesgue measure.
\end{theoremB}

Here and throughout, we use the notion of Palis attractor specified in Definition~\ref{def:Palis}. Theorem~B shows that generic interval maps possess a canonical robust observable attractor, although they admit no Palis attractor. Thus, for a typical dynamical system, observable dynamics is naturally described by a set rather than by a single invariant measure.

We emphasize that the notions of attractor considered in this paper require neither Lyapunov stability nor the existence of an attracting neighborhood. These attractors are not required to support a physical (SRB) measure. In the purely continuous setting, Abdenur and Andersson~\cite{AA2013} showed that generic maps on compact manifolds admit no physical (SRB) measures, so no single invariant measure describes the asymptotic statistics on a set of positive Lebesgue measure. Nevertheless, our results show that the absence of physical measures does not prevent the existence of a canonical observable attractor. Despite the absence of physical measures and Lyapunov stability, observable dynamics still recovers the entire recurrent dynamics. The resulting attractor is a zero-Hausdorff-dimensional Cantor set. Solenoidal Cantor dynamics in the smooth interval setting was studied, for example, by Blokh and Lyubich~\cite{BL1990}. Similar solenoidal structures also play a central role in our construction.

The paper is organized as follows.
Section~2 introduces the attractor notions used throughout the paper and establishes the inclusions
$
\Amin\subseteq A_{\mathrm{stat}}\subseteq\Lambda\subseteq\Omega.
$
Section~3 contains the proof of the generic identity
$
\Lambda(f)=\overline{\operatorname{Per}(f)}
=\overline{\operatorname{Rec}(f)}
=\Omega(f).
$
Section~4 presents examples showing that, outside the generic setting, the global Milnor attractor need not contain all recurrent dynamics and may have zero topological entropy while positive entropy is carried by a disjoint invariant Cantor set. Finally, Section~5 gives examples in which the Milnor, statistical, and physical attractors do not coincide. In particular,
$
\Amin\subsetneq A_{\mathrm{stat}}\subsetneq\Lambda.
$

\section{Preliminaries}
\subsection{Basic definitions and properties}
Throughout the main part of the paper, $X=[0,1]$, $f\colon X\to X$ is continuous, and $\lambda$ denotes Lebesgue measure. We use standard notation for orbits and recurrence. The \textit{orbit} of $x$ is
\[
\operatorname{Orb}_f(x)=\{f^n(x):n\ge 0\},
\]
and its \textit{$\omega$-limit set} is
\[
\omega_f(x)=\bigcap_{n\ge0}\overline{\{f^m(x):m\ge n\}}.
\]
When the map is clear, we write $\operatorname{Orb}(x)$ and $\omega(x)$. A point is \textit{recurrent} if $x\in\omega_f(x)$ and \textit{non-wandering} if every neighborhood $U$ of $x$ satisfies $f^n(U)\cap U\neq\emptyset$ for some $n>0$. We write $\Fix(f)$, $\Per(f)$, $\Rec(f)$, and $\Omega(f)$ for the sets of fixed, periodic, recurrent, and non-wandering points, respectively. The least positive $p$ with $f^p(x)=x$ is the \textit{(prime) period} of $x$. Throughout, $\mathbb N$ denotes the positive integers. Finally, $c\in[0,1]$ is called a \textit{critical} point if $f$ is not locally injective at $c$, equivalently, if $f$ is not strictly monotone on any neighborhood of $c$.

\subsection{Measures}\label{sec:measures}
Let $(X, \mu)$ be a measure space and $f \colon X \to Y$ a continuous function. The \textit{pushforward} of the measure $\mu$ under $f$ is defined as $f_* \mu = \mu \circ f^{-1}$, which is a measure on $Y$ capturing the distribution of $f$-images of $\mu$-typical points.

In our context, we are particularly interested in the case where $X = Y$ and $f \colon X \to X$ is a continuous transformation of a compact metric space. Given a reference measure, typically the Lebesgue measure $\lambda$, we define the sequence of pushforward measures $\mu_n = f^n_* \lambda$ and consider the averages
$$
\mu_n^{\text{avg}} = \frac{1}{n} \sum_{i=0}^{n-1} \mu_i = \frac{1}{n} \sum_{i=0}^{n-1} f^i_* \lambda.
$$
By the Krylov--Bogolyubov procedure, any weak-$^*$ accumulation point of the sequence $(\mu_n^{\mathrm{avg}})$ is an $f$-invariant probability measure. These limiting measures form the basis for the definition of the \emph{physical} (or \emph{minimal}) attractor, which is given by the closure of the union of the supports of all such limiting measures. An equivalent characterization of the physical attractor will be introduced later via the notion of essential sets.

\subsection{Limit sets}
Recall that the $\omega$-limit set of a point $x\in X$ can equivalently be written as
$$
\omega(x) =\{z\in X:\limsup_{n\to\infty}\chi_U(f^n(x))>0, \forall U\in Nb_z\},
$$
where $\chi_U$ denotes the indicator function of the set $U$, and $Nb_z$ is the system of all neighborhoods of a point $z$. \\
We define the $\sigma$-limit set of a point $x\in X$ as
$$\sigma(x) = \{z\in X:\limsup_{n\to\infty}\frac{1}{n}\sum_{k=0}^{n-1}\chi_U(f^k(x))>0, \forall U\in Nb_z\}.$$
The $\sigma$-limit set is also known as the \emph{minimal center of attraction} \cite{Sigmund} or simply the \emph{statistical limit set} \cite{Sharkovsky}.\\
An $\omega$-limit set $\omega(x)\subset[0,1]$ is called a \emph{solenoid} if it can be written in the form $$\omega(x)=\bigcap_{i=1}^{\infty}\bigcup_{j=1}^{k_i}I_{i,j}$$
where, for each $i\in\mathbb N$, the intervals $I_{i,1},\ldots,I_{i,k_i}$ are closed, pairwise disjoint and nondegenerate, $k_i\to\infty$, and
\begin{equation}
    \bigcup_{j=1}^{k_i}I_{i,j}\supset \bigcup_{j=1}^{k_{i+1}}I_{i+1,j}\text{, for every }i\in\mathbb{N},
\end{equation}

\begin{equation}
    f(I_{i,j})\subset I_{i,j+1},\qquad I_{i,k_i+1}:=I_{i,1},
    \quad i\in\mathbb N,\ j=1,\ldots,k_i,
\end{equation}
\begin{equation}\label{diam}
    \lim_{i\to\infty}\max_j\operatorname{diam} I_{i,j}=0.
\end{equation}

Note that the condition~\eqref{diam} is sometimes omitted, in which case $\omega(x)$ is only required to be covered by the nested intersection of periodic intervals
$\bigcap_{i=1}^{\infty} \bigcup_{j=1}^{k_i} I_{i,j}$.
See, for example, Blokh \cite{Blokh}. However, in the case when $\bigcap_{i=1}^{\infty} \bigcup_{j=1}^{k_i} I_{i,j}$ has a nonempty interior, any interval contained in it is a wandering interval. We exclude the possibility of a wandering interval in our construction by requiring condition~\eqref{diam}.\\
The following result, stated as Lemma 1 in \cite{Miz}, establishes that, for solenoids, the statistical and the $\omega$-limit sets coincide:
\begin{lem}{\cite[Lemma 1]{Miz}}\label{lem:solenoid} If $\omega(x)$ is a solenoid, then $\sigma(x)=\omega(x).$
\end{lem}
\subsection{Attractors}
For any nonempty set $U \subset X$, the \emph{basin of attraction} of $U$ is defined as the set of all points $x \in X$ for which the $\omega$-limit set is contained in $U$, that is,
$$
B(U) = \{ x \in X : \omega(x) \subseteq U \}.
$$
There are generally two perspectives when defining attractors in dynamical systems. A \emph{local attractor} describes the long-term behavior of a set of initial conditions of positive (Lebesgue) measure, whereas a \emph{global attractor} governs the behavior of almost all initial conditions, with its basin of attraction covering the entire space up to a set of measure zero. While the following definitions may not always distinguish explicitly between local and global versions, this distinction can typically be made in context.

\begin{defn}[Milnor local attractor]
A Milnor local attractor $A \subset X$ is a closed invariant set with $\lambda(B(A))>0$, such that there is no closed proper subset $A' \subset A$ for which $B(A)$ coincides with $B(A')$ up to a set of measure zero.
\end{defn}

\begin{defn}[Milnor global  attractor, likely limit set]
The Milnor global attractor, or the likely limit set, $\Lambda(f) \subset X$ of $f$ is the smallest closed set containing the $\omega$-limit sets of $\lambda$-almost all orbits.
 \end{defn}
The definition of a Palis attractor was introduced in \cite{Palis}. However, in this paper, we adopt the formal version as stated in \cite{Bonatti}.

\begin{defn}[Palis attractor]\label{def:Palis}
A closed invariant set $A \subset X$ is called a Palis attractor if:
\begin{enumerate}
   \item there exists an ergodic invariant probability measure $\mu$ such that $\mathrm{supp}(\mu) = A$,
    \item the statistical basin $B_\mu$ has positive Lebesgue measure, where
    $$
    x \in B_\mu \quad \Longleftrightarrow \quad \lim_{n \to \infty} \frac{1}{n} \sum_{i=0}^{n-1} \delta_{f^i(x)} = \mu,
    $$
    with $\delta_z$ denoting the Dirac measure at $z$.
\end{enumerate}
\end{defn}

\begin{defn}[Statistical attractor]\label{def:Astat}
The statistical attractor $A_{\mathrm{stat}}(f)$ of $f$
is the smallest closed subset $A\subset X$ with the
following property: for every open neighborhood $U$ of $A$,
$$
\lim_{n\to\infty}\frac{1}{n}
\sum_{k=0}^{n-1}\chi_U(f^k(x))=1
$$
for $\lambda$-almost every $x\in X$.
\end{defn}
The following theorem shows how to characterize the statistical attractors using $\sigma$-limit sets in the same way as the Milnor global attractors are defined using $\omega$-limit sets. The idea is derived from the work on continuous flows by Karabacak and Ashwin~\cite[Theorem~2.3]{Ashwin}.
\begin{thm}\label{def_stat}
The statistical attractor is the smallest closed set that contains the $\sigma$-limit sets of almost all trajectories.
\end{thm}
We also record the following elementary lemma, adapted from \cite{Ashwin}, and include its proof.

\begin{lem}\label{lem:Ashwin}
    Let $X$ be a compact metric space and $\gamma$ be any map from $X$ to the set of closed subsets of $X$. Let $\{U_i\}$ be a countable base for $X$. For any $A\subset X$, the following statements are equivalent:
    \begin{enumerate}
        \item\label{lem:Ashwin:1} $A^C=\bigcup \{ U_i: U_i\cap \gamma(x)=\emptyset \text{ for $\lambda$-almost every }x\in X\}$,
        \item\label{lem:Ashwin:2} $A$ is the smallest closed subset of $X$ that contains $\gamma(x)$ for $\lambda$-almost every point in $X$.
    \end{enumerate}
\end{lem}
\begin{proof}
Let $N$ denote the union in \eqref{lem:Ashwin:1}. Since the base
$\{U_i\}$ is countable, $\gamma(x)$ avoids $N$
for almost every $x$, and so $X\setminus N$ is
closed and contains $\gamma(x)$ for $\lambda$-almost every point in $X$.
Let $Q$ be any closed subset of $X$ containing
$\gamma(x)$ for $\lambda$-almost every $x\in X$.
For every $z\in X\setminus Q$, there exists $U_i$
such that $z\in U_i\subset X\setminus Q$.
Since $\gamma(x)\subset Q$ for almost every $x$,
we have $U_i\subset N$. Hence $X\setminus Q\subset N$, and therefore $X\setminus N$
is the smallest closed subset of $X$ containing
$\gamma(x)$ for $\lambda$-almost every $x\in X$.
Thus \eqref{lem:Ashwin:2} holds if and only if
$A=X\setminus N$, which is precisely
\eqref{lem:Ashwin:1}.
\end{proof}
\begin{proof}[Proof of Theorem \ref{def_stat}]
Let $\{U_i\}_{i\ge1}$ be a countable base for $X$, and define
$$
A=X\setminus\bigcup\left\{U_i:
U_i\cap\sigma(x)=\emptyset
\text{ for $\lambda$-almost every }x\in X\right\}.
$$
By Lemma~\ref{lem:Ashwin}, $A$ is the smallest closed
subset of $X$ containing $\sigma(x)$ for
$\lambda$-almost every $x\in X$. For any open neighborhood $V$ of $A$ and almost
every $x$, each point of the compact set
$X\setminus V$ has a neighborhood visited by the
orbit of $x$ with asymptotic frequency zero.
Since $X\setminus V$ is compact, finitely many such
neighborhoods cover it. As the orbit visits each of
them with asymptotic frequency zero, it also visits
$X\setminus V$ with asymptotic frequency zero,
and hence $V$ with asymptotic frequency one.

Conversely, let $Q\subset X$ be any closed set such that,
for every open neighborhood $U$ of $Q$, the orbit of
$\lambda$-almost every $x\in X$ visits $U$ with
asymptotic frequency one.
Consider
\[
V_m=\{z\in X:\operatorname{dist}(z,Q)<1/m\},
\qquad m\geq1.
\]
For almost every $x$, the orbit visits every $V_m$
with asymptotic frequency one. Hence every $z\notin Q$ has a neighborhood visited
with asymptotic frequency zero, so
$\sigma(x)\subset Q$ almost everywhere.
Applying Lemma~\ref{lem:Ashwin} with
$\gamma(x)=\sigma(x)$, we obtain $A\subset Q$.
Thus $A$ is the statistical attractor.
\end{proof}

To give a convenient formulation of the physical attractor, we call an open set $U \subset X$ an \emph{essential set} if there exist a sequence $\{k_n\}_{n \geq 0}$ with $k_n \to +\infty$ and constants $\varepsilon, \delta > 0$ such that for every $n \in \mathbb{N}$,
$$
\lambda\left( \left\{ x \in X : \frac{1}{k_n} \sum_{i=0}^{k_n - 1} \chi_U(f^i(x)) > \varepsilon \right\} \right) > \delta.
$$
In other words, an open set $U$ is \emph{not essential} if the time averages $\frac{1}{k} \sum_{i=0}^{k-1} \chi_U \circ f^i$ converge to zero in measure with respect to $\lambda$.

We now turn to the concept of \emph{physical attractors}, which encode the observable behavior of Lebesgue-typical orbits. As discussed in Section~\ref{sec:measures}, the Krylov--Bogolyubov procedure applied to the pushforward sequence $(f^n_* \lambda)$ produces $f$-invariant probability measures as weak$^*$ accumulation points of the time averages
$$
\mu_n^{\mathrm{avg}} = \frac{1}{n} \sum_{i=0}^{n-1} f^i_* \lambda.
$$
The physical attractor is defined in \cite{GorIlya, Ilyashenko4} as the smallest closed set containing the supports of all such accumulation points. Rather than working directly with these measures, we use the equivalent formulation via essential sets, following \cite{Bach}, which is more convenient for our purposes.
\begin{defn}[Physical attractor]\label{def:Amin}
The physical attractor $\Amin(f)$ is the set
$$
\Amin(f) = X \setminus \left\{ \bigcup W : W \text{ is open and not essential} \right\}.
$$
\end{defn}

For completeness, we remark that $\Astat(f)$ and $\Lambda(f)$ may also be defined in a similar way; see also \cite{Ilyash3}. For the statistical attractor, $$\Astat(f) = X\setminus \left\{\bigcup W: W \text{ is open and not statistically essential}\right\},$$ where an open set $U$ is not statistically essential if $\frac{1}{k}\sum_{i=0}^{k-1}\chi_U\circ f^i\to 0$ for $\lambda$-almost all  $x\in X$. For the Milnor attractor, $$\Lambda(f) = X\setminus \left\{\bigcup W: W \text{ is open and not Milnor essential}\right\},$$ where an open set $U$ is not Milnor essential if $\chi_U\circ f^k\to 0$ for $\lambda$-almost all  $x\in X$.

It follows that $\Amin(f)$ is the maximal closed set in $X$ such that any neighborhood of an arbitrary point in $\Amin(f)$ is an essential open set. \\

When $f$ is clear from the context, we denote the above attractors by $\Lambda$, $\Astat$ and $\Amin$, respectively. We recall the following fact from \cite{Ilyash3}:
\begin{lem}\label{lem:attr}
For any continuous map $f:X\to X$, we have
$$
\Amin(f)\subseteq \Astat(f)\subseteq \Lambda(f)\subseteq \Omega(f).
$$
\end{lem}

\begin{defn}[Stability]
    A non-empty closed set $A$ with $f(A)=A$ is said to be \textit{(Lyapunov) stable} if, for any open set $U\supset A$,
   there is an open set $V\supset A$ such that
   $\Orb_f(x)\subset U$ for every $x\in V$.
   
   A stable set $A$ is \textit{asymptotically stable} if, in addition, there is an open set $W\supset A$ contained in its basin of attraction, that is, $W\subset B(A)$.
\end{defn}

It is well known that $A$ is asymptotically stable if and only if there exists an arbitrarily small open neighborhood $U\supset A$ such that $f(\overline U)\subset U$
and $\bigcap_{n} f^n(U)=\bigcap_{n} f^n(\overline{U})=A$. The reader is referred to Chapter V in \cite{BC} for some basic properties of asymptotically stable sets (cf. \cite{Milnor}).

\begin{defn}[Robustness]
Let $A(g)$ be a compact invariant set associated with maps $g$ in a neighborhood of $f$. We say that the family $A(\cdot)$ is \emph{robust at $f$} if
$$
d_H(A(g),A(f))\longrightarrow 0
\qquad\text{whenever } g\to f
$$
in the uniform topology on $C(X,X)$.
\end{defn}

Equivalently, the family $A(\cdot)$ is robust at $f$ if $f$ is a continuity point of the set-valued map
$$
\Gamma:C(X,X)\longrightarrow K(X),\qquad
\Gamma(g)=A(g),
$$
where $K(X)$ denotes the space of nonempty compact subsets of $X$ equipped with the Hausdorff metric.

\begin{exmp}
    Consider the family of maps $f_n$ for $n\geq 2$ given for $x \in [0,1]$ by
$
f_n(x) = x^{\frac{n-1}{n}}.$
As $n \to \infty$, we have $f_n \to f = Id$ uniformly on $[0,1]$. Therefore, $\Lambda(f) = [0,1],$ since $f$ is the identity map.
For each $n$, observe that $\Lambda(f_n) = \{1\}.$ 
Thus, $\Lambda(f)$ is not robust, and the map $\Gamma(f) = \Lambda(f)$ is not continuous at $f$.
\end{exmp}

The following tools will be used to show that our constructed typical dynamical system has a robust Milnor attractor.
\begin{lem}\label{lem:solenoid-milnor}
Let $f\colon [0,1]\to[0,1]$ be continuous
and let $S\subset[0,1]$ be a solenoid.
Then $S\subset\Lambda(f)$.
\end{lem}

\begin{proof}
Fix $z\in S$ and $\varepsilon>0$.
By the definition of a solenoid, there exists
a nondegenerate interval $J\ni z$ and an integer
$k\geq1$ such that
$\operatorname{diam}(J)<\varepsilon$
and $f^k(J)\subset J$.

Since $B(\Lambda(f))$ has full Lebesgue measure,
choose $y\in J\cap B(\Lambda(f))$.
The subsequence $f^{nk}(y)$ stays in $J$,
so compactness gives
\[
\emptyset\ne\omega_f(y)\cap J
\subset\Lambda(f)\cap J.
\]
Hence $\operatorname{dist}(z,\Lambda(f))<\varepsilon$.
As $\varepsilon>0$ is arbitrary and $\Lambda(f)$
is closed, we obtain $z\in\Lambda(f)$.
\end{proof}

\begin{lem}\label{lem:USC} 
Let $f \in C([0,1])$, and suppose there exists a solenoid $S = \omega(x)$ defined by
$$
S = \bigcap_{i=1}^\infty \bigcup_{j=1}^{k_i} I_{i,j},
$$
where:
\begin{itemize}
    \item each $I_{i,j} \subset [0,1]$ is a closed, nondegenerate interval,
    \item $f(I_{i,j}) \subset \operatorname{int} I_{i,1+(j\bmod k_i)}$ for all $i, j$,
    \item $\max_j \operatorname{diam}(I_{i,j}) \to 0$ as $i \to \infty$.
\end{itemize}

Let $\{f_n\}_{n\in\mathbb{N}} \subset C([0,1])$ be a sequence converging 
uniformly to $f$, and let $\Lambda(f_n)$ denote the Milnor attractor of $f_n$ for 
each $n \in \mathbb{N}$. Then:
$$
S \subset \liminf_{n \to \infty} \Lambda(f_n).
$$
\end{lem}

\begin{proof}
Fix $z\in S$ and $\varepsilon>0$. Choose a level-$i$
interval $J=I_{i,j}$ containing $z$ such that
$\operatorname{diam}(J)<\varepsilon$.
The strict cyclic inclusions imply
$f^{k_i}(J)\subset\operatorname{int}J$.
Since $f_n^{k_i}\to f^{k_i}$ uniformly, we have
$f_n^{k_i}(J)\subset J$ for all sufficiently large $n$.

For each such $n$, choose
$y_n\in J\cap B(\Lambda(f_n))$, which is possible
since $J$ has positive Lebesgue measure and
$B(\Lambda(f_n))$ has full measure.
The subsequence $f_n^{\ell k_i}(y_n)$ stays in $J$,
so compactness gives
$$\emptyset\ne\omega_{f_n}(y_n)\cap J
\subset\Lambda(f_n)\cap J.$$

Hence $\operatorname{dist}(z,\Lambda(f_n))<\varepsilon$
for all sufficiently large $n$. Therefore,
$$
S\subset\liminf_{n\to\infty}\Lambda(f_n)
$$
which completes the proof.
\end{proof}

\section{Properties of the typical attractor on the interval}
The construction builds on the perturbative scheme of Mizera~\cite[Theorem~1]{Miz}, but with a different objective. Mizera's argument controls the $\omega$-limit sets of Lebesgue-almost every point. By itself, this does not exclude additional recurrent subsystems of zero Lebesgue measure. Our main new ingredient is simultaneous control of the periodic points created throughout the approximation, linking the observable nested structure to the entire recurrent dynamics.

The proof proceeds in four steps. First, the construction produces a full-measure set of points with solenoidal $\omega$-limit sets and shows that their union generates the global Milnor attractor $\Lambda(g)$. Second, the periodic-point control yields
$
\Lambda(g)=\overline{\operatorname{Per}(g)}.
$
Third, the constructed maps satisfy the shadowing property. Since, for interval maps,
$
\overline{\operatorname{Per}(g)}=\overline{\operatorname{Rec}(g)},
$
shadowing implies
$
\overline{\operatorname{Rec}(g)}=\Omega(g)
$
(see \cite[Corollary 6.7]{KOU} for a direct conclusion).
In the fourth step, we combine these identities to obtain
$$
\Lambda(g)=\overline{\operatorname{Per}(g)}
=\overline{\operatorname{Rec}(g)}
=\Omega(g).
$$

In the following proofs, we write $d(x,y)$ for the Euclidean distance on $[0,1]$, and
$\|f-g\|_{\infty}=\sup_{x\in[0,1]}d(f(x),g(x))$ for the $C^0$ metric on $C([0,1])$.
For compact sets $A,B\subset[0,1]$, we denote by $d_H(A,B)$ the Hausdorff distance.

We begin by introducing the combinatorial structure of intervals that allows precise control of some properties of perturbations. These intervals, following the notation of \cite{Miz}, provide a partition of the unit interval into small building blocks with controlled diameters.
\subsection*{Notation}
For $m \in \N$, define
\begin{align*}
M &= 2^m, \quad \eta = \eta_m = 2^{-2m(m+1)}\\
a_i & = \frac{i}{M}, \quad i=0,\dots,M\\
c_i &= \frac12(a_{i-1}+a_i), \quad i=1,\dots,M\\
E^i_m&= [a_{i-1}+\eta, a_i-\eta],\quad F^i_m=[c_i-\eta, c_i+\eta],\quad i=1,\dots,M\\
H^i_m &= (a_i-\eta, a_i+\eta),\quad i=1,\dots,M-1\\
H^0_m &= [0,\eta), \quad H^M_m=(1-\eta,1]
\end{align*}
\begin{align*}
E_m&=\bigcup_{i=1}^ME^i_m,\quad F_m=\bigcup_{i=1}^MF^i_m,\quad H_m = \bigcup_{i=0}^MH^i_m\\
L_m&=\bigcup_{i=0}^{M/2-1}E^{2i+1}_m, \quad G_m=\bigcup_{i=0}^{M/2-1}F^{2i+1}_m\\
R_m&=\bigcup_{i=1}^{M/2}E^{2i}_m,\quad D_m=\bigcup_{i=1}^{M/2}F^{2i}_m\\
C_m&=\{c_1,c_2,\dots,c_M\}
\end{align*}
\begin{center}
The structure of these blocks is shown in the figure below:
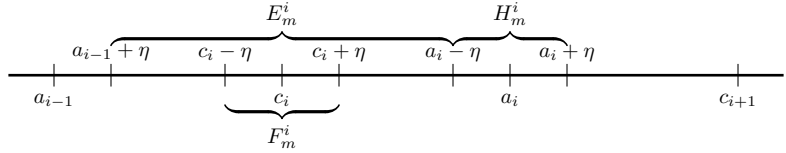
\begin{tikzpicture}[scale=1.5,>=stealth,
every node/.style={align=center,scale=.8}] 
\draw[line width=1pt] (.6,0)--(7.4,0);

\draw (1,.1)--(1,-.1) node[below]{$a_{i-1}$};
\draw (1.5,.1)--(1.5,-.1) node[above=3mm]{{$a_{i-1}+\eta$}};
\draw[B] (1.5,.1) -- node[above=5mm] {$E^i_m$} (4.5,.1);
\draw (2.5,.1)--(2.5,-.1) node [above=3mm]{$c_i-\eta$};
\draw (3,.1)--(3,-.1) node[below]{$c_i$};
\draw (3.5,.1)--(3.5,-.1) node [above=3mm]{$c_i+\eta$};
\draw[C] (2.5,-.1) -- node[below=5mm] {$F^i_m$} (3.5,-.1);
\draw (4.5,.1)--(4.5,-.1) node[above=3mm]{{$a_i-\eta$}};
\draw (5,.1)--(5,-.1) node[below]{$a_i$};
\draw (5.5,.1)--(5.5,-.1) node[above=3mm]{{$a_i +\eta$}};
\draw[B] (4.5,.1) -- node[above=5mm] {$H^i_m$} (5.5,.1);
\draw (7,.1)--(7,-.1) node[below]{$c_{i+1}$};
\end{tikzpicture}
\end{center}    
The sets defined above have the following properties:
\begin{enumerate}
\item the intervals $E^i_m$ and $H^i_m$ together form a partition of $[0,1]$,
\item $F^i_m\subset E^i_m$, $D_m\subset R_m$, $G_m\subset L_m$,
\item $\diam H^i_m<2^{-m}$ and $\sum_{i=1}^M (\diam H^i_m)^{\frac1m}<2^{-m}$,
\item $\diam F^i_m<2^{-m}$ and $\sum_{i=1}^M (\diam F^i_m)^{\frac1m}<2^{-m}$,
\item if $n\geq 2m(m+1)$, then for every $i$, the intersections $F^i_m\cap G_n$, $F^i_m\cap D_n$ are nonempty and consist of precisely those components of $G_n$ or $D_n$ which have nonempty intersections with $F^i_m$.
\item $L_m\cap R_m=D_m\cap G_m=\emptyset$.
\end{enumerate}

Note that, for a fixed $m \in \mathbb{N}$ and every $i = 1,\dots,M-1$, we have
$$
\diam E^i_m = \frac1M-2\eta =
\frac{2^{2m^2+m}-2}{2^{2m(m+1)}}
\quad \text{ and }\quad \diam H^i_m = 2\eta = \frac{2}{2^{2m(m+1)}}.
$$
From the definition, it is clear that $\diam E^i_m+\diam H^i_m=\frac{1}{M}$.
\medskip

We now state the technical theorem underlying the construction. It provides the generic framework from which the main results will be derived.

\begin{thm}\label{thm:residual}
There exists a residual set $G\subset C([0,1])$ such that, for every $g\in G$, there exists a set $Z(g)\subset[0,1]$ satisfying the following conditions: 
\begin{enumerate}[(a)]
\item the set $Z(g)$ has full Lebesgue measure, \label{thm:a}
\item\label{thm:a'} the set $g^{-1}(Z(g))$ has zero Lebesgue measure,
\item $\omega(x)$ is a solenoid for every $x\in Z(g)$, \label{thm:b}
\item for every $x\in Z(g)$, the basin of attraction $B(\omega(x))$ has zero Lebesgue measure. \label{thm:c}
\end{enumerate}
Moreover, the set $A(g)=\overline{\bigcup_{x \in Z(g)}\omega(x)}$ has the following properties:
\begin{enumerate}[resume*]
\item $A(g)$ has zero Hausdorff dimension, \label{thm:d}
\item $A(g)=\overline{\Per(g)}=\Omega(g)$ and consequently the basin of attraction $B(A(g))=[0,1]$,\label{thm:e}
\item every open set $U$ with $U \cap A(g) \neq \emptyset$ is essential. \label{thm:g}
\end{enumerate}
\end{thm}
\begin{proof}
We start the proof with the construction of the set $G$ and consequently the set $Z(g)$ for $g\in G$.

For every $K\geq 1$, let $\mathcal{F}_K\subset C ([0,1])$ be a dense set of piecewise linear maps whose affine pieces have slopes of absolute value at least $3$ and such that, for every $f \in \mathcal{F}_K$ and $k\leq K$, no critical point is $k$-periodic, $f^k(0)\neq 0$, and $f^k(1)\neq 1$. For technical convenience, we regard each $\mathcal{F}_K$ as a sequence, and write $\mathcal{F}_K = \{f_{j,K}: j\geq 1\}$ for every $K$.

From now on, we proceed with the construction of an open dense set $G_K\subset C([0,1])$, fixing any $K\geq 1$, an arbitrary map $f_{j,K} \in \mathcal{F}_K$, and some fixed $\gamma_{j,K}<\frac{1}{4j}$. For simplicity of notation, we write $f=f_{j,K}$ and $\gamma = \gamma_{j,K}$. All further constants will depend on $f$. While we do not write this dependence directly (e.g. we write $m_K$ instead of $m_K^f$, etc.), the reader should keep in mind that such dependence exists.

Consider the finite set $\Fix(f^K)$, and denote by
$\Crit(f^K)$ the set of critical points of $f^K$
together with $0$ and $1$.
By the choice of $f\in\mathcal F_K$, we have
$\Fix(f^K)\cap\Crit(f^K)=\emptyset$.
Indeed, the endpoints are not fixed by $f^K$,
and if a fixed point of $f^K$ were critical,
its periodic orbit would contain a critical
point of $f$ of period at most $K$.
Moreover, on each monotonicity interval of $f^K$,
the slopes have absolute value at least $3^K$,
so such an interval contains at most one fixed point.

Let $\theta_f$ be the maximum absolute value of
the slopes of $f$, and hence a Lipschitz constant for $f$.
If necessary, we further decrease $\gamma$ so that
for every $h\in C([0,1])$,
\begin{equation}\label{ieq:gamma2}
\|h-f\|_\infty<\gamma
\Rightarrow
d_H(\Fix(h^K),\Fix(f^K))<\frac{1}{4K}.
\end{equation}
Indeed, every fixed point of $f^K$ is interior
and noncritical. Since the slopes of $f^K$
have absolute value at least $3^K>1$,
the function $f^K(x)-x$ changes sign at each
fixed point. Hence these fixed points persist
under small uniform perturbations.
Moreover, $|f^K(x)-x|$ is bounded away from zero
outside small neighborhoods of $\Fix(f^K)$,
so no other fixed points can appear there.
Since $h^K\to f^K$ uniformly as $h\to f$,
the assertion follows.

Next, define
\[
\beta_K=\frac12\operatorname{dist}
\bigl(\Fix(f^K),\Crit(f^K)\bigr)>0.
\]

Choose $\alpha_K<\min\{\beta_K,\frac{\gamma}{2},\frac{1}{K}\}$ and $\delta_K<\alpha_K$ with the following properties:
\begin{enumerate}[(i)]
\item $d(x,y)<\delta_K$ implies $d(f^i(x),f^i(y))<\frac{\alpha_K}{4}$ for every $x,y\in [0,1]$, $i=1,\dots,K$,\label{pseodoorb1}
\item for every $\delta_K-$pseudo-orbit $\{x_0,\dots, x_{K}\}$ of length $K+1$, we have $$d(f^i(x_0), x_i)<\frac{\alpha_K}{4}, \; i=1,\dots,K.$$ \label{pseudoorb}
\end{enumerate}

We select $m_K>K$ large enough so that the following inequalities hold:
\begin{align}
4(\diam E^i_{m_K}+\diam H^i_{m_K})= \frac{4}{M_K}&<\frac{\delta_K}{4}, \label{ineq:diams}\\
M_K = 2^{m_K}&>\frac{4}{K\alpha_K}, \label{ineq:M_K}\\
\frac{2\theta_f+6}{M_K}&<\min\left\{\frac{\alpha_K}{4},\delta_K\right\}.\label{ineq:theta}
\end{align}
The condition \eqref{ineq:diams} ensures that 
$\diam(E^i_{m_K})<\delta_K$ for all $i$, so property \ref{pseodoorb1} applies on each $E^i_{m_K}$, 
and, for every $y \in \Fix(f^K)$, the $\delta_K-$neighborhood $(y-\delta_K, y+\delta_K)$ contains at least four consecutive sets from $\{E^i_{m_K}\}_{i=1}^{M_K}$.

We now construct a perturbation $\tilde f$ of $f$ following the approach of \cite[Theorem~1]{Miz}, with adjustments near periodic points in order to produce disjoint periodic orbits absorbed by the attracting set $F_{m_K}$. Define the map $\tilde f : [0,1] \to [0,1]$ by the following rules:
\begin{enumerate}[(M1)]
\item\label{con:M1} $\tf(c_i)$ is the element closest to $f(c_i)$ in $\begin{cases} D_{m_K}\cap C_{m_K}&, \text{ if }c_i \in D_{m_K}\\ G_{m_K}\cap C_{m_K}&, \text{ if }c_i \in G_{m_K}\end{cases}$,
\item $\tf(x) = \tf(c_i)$ for every $x \in E^i_{m_K}$, \label{def:M2}
\item $\tf$ is linear on every $H^i_{m_K}$ and $\tf$ is continuous,
\item $\tf(0)$ and $\tf(1)$ are the elements of $C_{m_K}$ closest to $f(0)$ and $f(1)$, respectively.
\end{enumerate}
In (M1) and (M4), whenever two centers are equally
close, we choose the smaller one. With this convention,
the map $\tilde f$ is uniquely determined.

We will now show that $\|f-\tf\|_{\infty}\leq \frac{\alpha_K}{4}$. Observe that $d(c_i,c_{i-1})=1/M_K$ and
centers of fixed parity are spaced by $2/M_K$.
Since the distance to the nearest center of a
prescribed parity is at most $3/(2M_K)$,
including near the endpoints, we obtain
\[
d(f(c_i),\tilde f(c_i))
\le \frac{3}{2M_K}\le\frac{2}{M_K}.
\]

Now, for every $x\in E^i_{m_K}$, we have
$d(x,c_i)<\frac{1}{2M_K}$, so
$d(f(x),f(c_i))\le\frac{\theta_f}{2M_K}$.
Hence, by \eqref{ineq:theta},
\[
\begin{aligned}
d(f(x),\tilde f(x))
&\le d(f(x),f(c_i))
   +d(f(c_i),\tilde f(c_i))\\
&\le \frac{\theta_f+4}{2M_K}
<\frac{\alpha_K}{4}.
\end{aligned}
\]

Similarly, for every $x \in H^i_{m_K}$, we have
$x\in [a_i-\eta,a_i+\eta]\subset(c_i,c_{i+1})$
and $d(c_i,c_{i+1})=\frac{1}{M_K}$.
As $\tilde f$ is constant on each of $E^i_{m_K}$
and $E^{i+1}_{m_K}$ and linear on $H^i_{m_K}$,
we obtain
\[
\diam(\tilde f(H^i_{m_K}))
=d(\tilde f(c_i),\tilde f(c_{i+1})).
\]
Using \ref{con:M1}, we have
$$
d(\tilde f(c_{i+1}),\tilde f(c_i))
\le d(f(c_{i+1}),f(c_i))
+ d(f(c_{i+1}),\tilde f(c_{i+1}))
+ d(f(c_i),\tilde f(c_i))
\le \frac{\theta_f+4}{M_K}.
$$
Thus,
$$
d(\tilde f(x),\tilde f(c_i))
\le \diam(\tilde f(H^i_{m_K}))
\le \frac{\theta_f+4}{M_K},
$$
while
$d(x,c_i)<\frac{1}{M_K}$ implies
$d(f(x),f(c_i))\le \frac{\theta_f}{M_K}$.
Therefore,
$$
d(f(x),\tilde f(x))
\le \frac{\theta_f}{M_K}+\frac{2}{M_K}
+\frac{\theta_f+4}{M_K}
=\frac{2\theta_f+6}{M_K}.
$$
As a consequence, by \eqref{ineq:theta},
$$
d(f(x),\tilde f(x))
\le \frac{2\theta_f+6}{M_K}
< \frac{\alpha_K}{4}.
$$
The same bound holds on the endpoint intervals
$H^0_{m_K}$ and $H^{M_K}_{m_K}$, by linear
interpolation between the endpoint and adjacent
center values, using the rounding bound $2/M_K$
and the Lipschitz constant $\theta_f$.
Therefore,
\begin{equation}
\|f-\tilde f\|_{\infty}
\le \frac{2\theta_f+6}{M_K}
<\min\left\{\frac{\alpha_K}{4},\delta_K\right\}.
\label{eq:tf}
\end{equation}
Now we introduce a further perturbation of $\tf$
to obtain a map $\hat f$ whose periodic points
trace the periodic orbits of $f$ contained in
$\Fix(f^K)$.

Choose one representative $y$ of each distinct
periodic orbit in $\Fix(f^K)$.
Let $p$ be its prime period and choose a center
$c_{i_0}$ closest to $y$.
Set
$$
c_{i_j}=\tf^j(c_{i_0}),\qquad j=0,\dots,p.
$$
These are centers because $\tf(C_{m_K})\subset C_{m_K}$.

Since $\|f-\tf\|_\infty<\delta_K$,
the sequence $x_j=\tf^j(c_{i_0})$, $j=0,\dots,K$,
is a $\delta_K$-pseudo-orbit of $f$.
Moreover,
$d(c_{i_0},y)\le 1/(2M_K)<\delta_K$.
By \ref{pseudoorb} and \ref{pseodoorb1}, we have
$$
d(c_{i_j},f^j(y))<\frac{\alpha_K}{2},
\qquad j=0,\dots,p.
$$
Since distinct fixed points lie in different
monotonicity intervals of $f^K$, a point of
$\Crit(f^K)$ lies between them. Hence distinct
points of $\Fix(f^K)$ are separated by more
than $4\alpha_K$, so the indices
$i_0,\dots,i_{p-1}$ are distinct, and the
corresponding indices for different periodic
orbits are disjoint.

In particular, since $f^p(y)=y$, we obtain
$$
d(\tf^p(c_{i_0}),y)<\frac{\alpha_K}{2},
$$
and consequently
$$
d(\tf^p(c_{i_0}),c_{i_0})
<\frac{\alpha_K}{2}+\frac{1}{M_K}.
$$

For each selected periodic orbit, define $\hat f$ by
$$
\hat f(x)=
\begin{cases}
\tilde f(x), & x\notin
H^{i_{p-1}-1}_{m_K}\cup E^{i_{p-1}}_{m_K}
\cup H^{i_{p-1}}_{m_K},\\[4pt]
c_{i_0}, & x\in E^{i_{p-1}}_{m_K},
\end{cases}
$$
and interpolate linearly on
$$
H^{i_{p-1}-1}_{m_K}\cup H^{i_{p-1}}_{m_K}
$$
so that $\hat f$ remains continuous.
All these modifications are performed simultaneously.
Since the modified blocks are distinct, the assignments
are compatible. If a gap is adjacent to two modified
blocks, we interpolate between their new endpoint values.

By construction, we have
$\|\tilde f-\hat f\|_{\infty}
\le \frac{\alpha_K}{2}+\frac{1}{M_K}$
on each modified block $E^{i_{p-1}}_{m_K}$.
Since $H^{i_{p-1}-1}_{m_K}$ and
$H^{i_{p-1}}_{m_K}$ are the two adjacent gaps,
linear interpolation preserves the same bound,
even if both adjacent blocks are modified.
Therefore, the estimate holds on all modified blocks
and their adjacent gaps, and hence on the whole interval
$[0,1]$, since $\hat f=\tilde f$ elsewhere. Also, since $\tilde f$ preserves the parity of centers,
the modifications defining $\hat f$ preserve it
as well.

Observe that the above estimate, together with
\eqref{eq:tf} and \eqref{ineq:diams}, implies
$$
\|f-\hat{f}\|_{\infty}
\leq \|f-\tf\|_{\infty}+\|\tf-\hat{f}\|_{\infty}
\leq \frac{\alpha_K}{4}
+\left(\frac{\alpha_K}{2}+\frac{1}{M_K}\right)
\leq \alpha_K<\frac{\gamma}{2}.
$$
The point $c_{i_0}$ is a periodic point of $\hat f$
whose prime period $p$ is the same as the prime
period of $y$ (in particular, $p$ divides $K$),
and therefore
$c_{i_0}\in C_{m_K}\cap\Fix(\hat f^{\,K})$.
Hence, for every $y\in\Fix(f^K)$, we obtain
a periodic point $x\in C_{m_K}\cap\Fix(\hat f^{\,K})$
whose prime period coincides with that of $y$.
Moreover, if two points of $\Fix(f^K)$ have
disjoint $f$-orbits, then the corresponding
periodic $\hat f$-orbits constructed in $C_{m_K}$
are also disjoint.

At this stage, we have shown how to construct a perturbation $\hf$ of any map $f \in \mathcal{F}_K$ for all $K\geq 1$. We now proceed to define the residual set of functions $G$.

For every $K\geq1$ and every $f\in\mathcal F_K$,
choose 
\begin{equation}\label{def:rho}
\rho_{f,K}<\min\{\eta_{m_K},\frac12\diam E^i_{m_K}\}.
\end{equation}
Observe that by \eqref{ineq:diams} and the fact that $\delta_K<\alpha_K<\frac{\gamma}{{2}}$ we have $\rho_{f,K}<\frac{\gamma}{4}$. Moreover, by the choice of $\rho_{f,K}$, the following holds: if $ g \in B(\hf,\rho_{f,K})$ and $c \in C_{m_K}\cap \Fix(\hf^K)\cap E^{i_c}_{m_K}$ for some $i_c \in \{1,\dots,M_K\}$, then 
\begin{equation}\label{eq:fixg}
\Fix(g^K)\cap E^{i_c}_{m_K}\neq\emptyset.
\end{equation}
Indeed, since $\hat f$ maps each $E^i_{m_K}$
to a center $c_j$, the inequality
$\rho_{f,K}<\eta_{m_K}$ implies
$$
g(E^i_{m_K})\subset\Int F^j_{m_K}
\subset\Int E^j_{m_K}.
$$
Thus $g$ follows the same cycle of blocks as
$\hat f$, and hence
$$
g^K(E^{i_c}_{m_K})
\subset\Int F^{i_c}_{m_K}
\subset\Int E^{i_c}_{m_K}.
$$
By the intermediate value theorem,
$\Fix(g^K)\cap E^{i_c}_{m_K}\neq\emptyset$.

Define
$$
G_K = \bigcup_{f \in \mathcal{F}_K}B(\hat{f},\rho_{f,K})$$ 
and
$$
G = \bigcap_{K\geq 1}G_K.
$$
Each set $G_K$ is open by definition. It is dense in $C([0,1])$ because $\mathcal{F}_K$ is dense 
 and 
$$
\lim_{j\to \infty}\|f_{j,K}-\hat{f}_{j,K}\|_{\infty}\leq \lim_{j\to\infty }\gamma_{j,K}=0.
$$ 
In particular, $G$, as a countable intersection of dense open sets, is a residual subset of $C([0,1])$.

By \cite[Theorem 4]{Miz}, there is a residual set $E$ in $C([0,1])$ such that every $f\in E$ has the shadowing property. The intersection of residual sets remains residual. Hence, replacing $G$ with the set $G\cap E$
when necessary, we may assume that every map in $G$ has the shadowing property.
By Coven and Hedlund \cite{CoHed} or Sharkovsky \cite{Shark2}, we have, for any interval map, $\overline{\Per(f)}=\overline{\Rec(f)}$. It is well known (e.g. see \cite{Oprocha}) that $\overline{\Rec(f)}=\Omega(f)$ whenever $f$ has the shadowing property.
Combining these facts,  we have 
\begin{equation}\label{eq:omega}
\Omega(g)=\overline{\Per(g)} \quad \text{ for every } g\in G.
\end{equation}

We also intersect $G$ with the residual set given by
\cite{Bernardes2006}. Thus, for every $g\in G$,
there exists a Borel set $W_g$ of full Lebesgue
measure such that $\omega(x)$ is infinite for
every $x\in W_g$.

Fix any $g\in G$ and any $K\geq1$.
By the definition of $G$, there exists a map
$f_{j,K}\in\mathcal F_K$ with $j\geq1$ such that
$$
g\in B(\hat f_{j,K},\rho_{f_{j,K},K}).
$$
Then the following set is nonempty and contains
at least one pair with second coordinate $K$
for every $K\geq1$:
$$
\mathcal B_g=
\left\{(f,K): K\geq1,\ f\in\mathcal F_K,\;
g\in B(\hat f,\rho_{f,K})\right\}.
$$

For every $(f,K)\in\mathcal B_g$, the
corresponding $m_K=m_K^f$ is chosen as in
the above construction so that the partition
$$
\{E^i_{m_K}\}_{i=1}^{M_K}
\;\cup\;
\{H^i_{m_K}\}_{i=0}^{M_K}
$$
of $[0,1]$ satisfies conditions
\eqref{ineq:diams}--\eqref{ineq:theta}.

Since $g\in G$, for every $K\geq1$ there exists
$(f,k)\in\mathcal B_g$ such that $k>K$.
We now define
$$
Z(g)=\bigcap_{K\geq1}\;
\bigcup_{\{(f,k)\in\mathcal B_g:\ k\geq K\}}
E_{m_k^f}.
$$
The set $Z(g)$ is obtained as the intersection
of unions of the sets $E_{m_k^f}$ at arbitrarily
fine levels. Since $m_k^f>k$ and
$\lambda(E_m)\to1$, each union has full
Lebesgue measure. Hence $Z(g)$ has full
measure, proving \ref{thm:a}.

Since $g(E_m)\subset F_m$ at every construction
level, we have
$$
g(Z(g))\subset\limsup_{m\to\infty}F_m.
$$
Moreover,
$$
\sum_{m\geq1}\lambda(F_m)
=\sum_{m\geq1}2^{m+1-2m(m+1)}<\infty,
$$
so
$$
\lambda(g(Z(g)))=0.
$$
Replacing $Z(g)$ by
$$
Z'(g)=(Z(g)\cap W_g)\setminus g(Z(g)),
$$
where $W_g$ is the full-measure set on which
$\omega(x)$ is infinite, we still have
$\lambda(Z'(g))=1$, and
$$
g^{-1}(Z'(g))\subset[0,1]\setminus Z(g),
$$
so $\lambda(g^{-1}(Z'(g)))=0$,
proving \ref{thm:a'}.

From now on, we denote $Z'(g)$ again by
$Z(g)$ and define
$$
A(g)=\overline{\bigcup_{x\in Z(g)}\omega(x)}.
$$

For $x\in Z(g)$, choose an increasing sequence
$m_j\to\infty$ such that $x\in E_{m_j}$.
Since
$$
g(E_{m_j})\subset\Int F_{m_j}\subset E_{m_j},
$$
we have $g^n(x)\in F_{m_j}$ for every $n\ge1$
and every $j$. Hence, for every $r\ge1$,
$$
g^n(x)\in Q_r:=\bigcap_{j=1}^r F_{m_j}
\qquad\text{for every }n\ge1.
$$
Moreover, $g(Q_r)\subset\Int Q_r$.

For each $r$, choose the cycle of interval
components of $Q_r$ eventually visited by the
orbit of $x$, and denote it by
$I_{r,1},\dots,I_{r,k_r}$.
Then
$$
g(I_{r,j})\subset\Int I_{r,j+1},
\qquad I_{r,k_r+1}:=I_{r,1}.
$$
Since $Q_{r+1}\subset Q_r$, these cycles are
nested, and $k_r$ divides $k_{r+1}$.
Moreover,
$$
\max_j\diam I_{r,j}\le 2\eta_{m_r}\to0.
$$
Every component of the selected cycle meets
$\omega(x)$, and consequently
$$
\omega(x)=\bigcap_{r\ge1}
\bigcup_{j=1}^{k_r}I_{r,j}.
$$
Finally, since $x\in Z(g)\subset W_g$,
the set $\omega(x)$ is infinite.
If the periods $k_r$ were bounded, the above
representation and the diameter condition would
imply that $\omega(x)$ is finite, a contradiction.
Thus $k_r\to\infty$. This shows that $\omega(x)$ is a solenoid
for every $x\in Z(g)$, proving \ref{thm:b}.

To prove \ref{thm:c}, fix $x\in Z(g)$ and
the sequence $\{m_j\}$ chosen above.
Since
$$
g(L_{m_j})\subset G_{m_j}\subset L_{m_j},
\qquad
g(R_{m_j})\subset D_{m_j}\subset R_{m_j},
$$
the set $\omega(x)$ lies either in $G_{m_j}$
or in $D_{m_j}$. Hence
$$
B(\omega(x))\cap L_{m_j}=\emptyset
\quad\text{or}\quad
B(\omega(x))\cap R_{m_j}=\emptyset.
$$
Which of the two intersections is empty may depend on $j$.

Suppose that $\lambda(B(\omega(x)))>0$.
By the Lebesgue density theorem, there exists
a nondegenerate interval $[a,b]$ such that
$$
\lambda([a,b]\cap B(\omega(x)))>\frac78(b-a).
$$
Since $L_m$ and $R_m$ consist of alternating
blocks of asymptotically equal length,
$$
\lim_{m\to\infty}\lambda([a,b]\cap L_m)=\lim_{m\to\infty} \lambda([a,b]\cap R_m)=
\frac{b-a}{2}.
$$
Thus, for sufficiently large $j$,
$B(\omega(x))$ meets both $L_{m_j}$ and
$R_{m_j}$ in positive measure, a contradiction.

To prove \ref{thm:d}, note that at each construction
level $m=m_K^f$ one has
$g(E_m)\subset F_m\subset E_m$.
If $z\notin F_m\cup\overline{H_m}$, then $z$
has a neighborhood $U\subset E_m\setminus F_m$.
Since $g^n(U)\subset F_m$ for every $n\ge1$,
we have $z\notin\Omega(g)$. Therefore,
$$
A(g)\subset\Omega(g)\subset
F_m\cup\overline{H_m}.
$$
This set is covered by the $M$ intervals $F_m^i$
and the closures of the $M+1$ intervals $H_m^i$,
each of diameter at most
$2\eta_m=2^{\,1-2m(m+1)}$.

Thus, for any $s>0$ and arbitrarily large
construction levels $m$, we have
$$
\mathcal{H}^s_\infty(A(g))
\le (2M+1)(2\eta_m)^s
\le 3\cdot 2^{\,m+s-2s\,m(m+1)}
\xrightarrow[m\to\infty]{}0.
$$
Hence $\mathcal{H}^s(A(g))=0$ for every
$s>0$, and therefore $\dim_H(A(g))=0$.

Next, we show that for every $g\in G$,
the attractor $A(g)$ coincides with the closure
of its periodic points, proving \ref{thm:e}. Fix $\varepsilon>0$ and $z_0\in\Per(g)$
with prime period $p$. Choose a sufficiently
large multiple $K$ of $p$ and
$(f,K)\in\mathcal B_g$.
Since
$\|g-\hat f\|_\infty<\rho_{f,K}<\gamma/4$
and $\|\hat f-f\|_\infty<\gamma/2$,
we obtain $\|g-f\|_\infty<\gamma$.
As $z_0\in\Fix(g^K)$, by \eqref{ieq:gamma2}
there exists $y\in\Fix(f^K)$ such that
$$
d(y,z_0)<\frac{1}{4K}.
$$
By the construction of $\hat f$, there exists
a center $c\in C_{m_K}\cap\Fix(\hat f^K)$
such that $d(c,y)<\alpha_K/2$.
Let $J=E^i_{m_K}$ be the interval containing $c$.
Since $\alpha_K<1/K$ and $m_K>K$,
for sufficiently large $K$ we have
$$
\sup_{w\in J}d(w,z_0)
\le\frac{1}{4K}+\frac{\alpha_K}{2}
+\diam J<\varepsilon.
$$
Moreover, by the trapping property established
above,
\begin{equation}\label{eq:periodic-trap}
J\subset B_\varepsilon(z_0),
\qquad
g^K(J)\subset\Int J.
\end{equation}

Since $J$ is nondegenerate and $Z(g)$ has
full Lebesgue measure, choose $x\in J\cap Z(g)$.
The subsequence $g^{nK}(x)$ stays in $J$,
so compactness gives
$$
\emptyset\ne\omega(x)\cap J
\subset A(g)\cap J.
$$
Hence $\dist(z_0,A(g))<\varepsilon$.
Since $\varepsilon>0$ is arbitrary and
$A(g)$ is closed, we obtain $z_0\in A(g)$.
Consequently,
$$
\overline{\Per(g)}\subset A(g).
$$

Since $g\in G$, by \eqref{eq:omega} we have
$\Omega(g)=\overline{\Per(g)}$.
For every $x\in Z(g)$ we thus obtain
$\omega(x)\subset\Omega(g)=\overline{\Per(g)}$,
hence
$$
A(g)=\overline{\bigcup_{x\in Z(g)}\omega(x)}
\subset\overline{\Per(g)}.
$$
Therefore $A(g)=\overline{\Per(g)}=\Omega(g)$,
establishing \ref{thm:e}.
Since $\Omega(g)=A(g)$, every orbit satisfies
$\omega(x)\subset A(g)$, and thus
$B(A(g))=[0,1]$.

To complete the proof of the theorem, fix any open set $U$ such that
$U\cap A(g)\neq\emptyset$.
Since $A(g)=\overline{\Per(g)}$, choose
$z_0\in U\cap\Per(g)$.
Choose $\varepsilon>0$ such that
$B_\varepsilon(z_0)\cap[0,1]\subset U$.
By \eqref{eq:periodic-trap}, there exist $K\ge1$
and a nondegenerate closed interval
$J\subset B_\varepsilon(z_0)\cap[0,1]\subset U$
such that
$$
g^K(J)\subset\Int J.
$$
Therefore, for every $x\in J$, the orbit of $x$
returns to $J\subset U$ at every multiple of $K$.
Consequently, for every $n\ge1$,
$$
\frac1n\sum_{j=0}^{n-1}\chi_U(g^j(x))
\ge\frac{\lfloor(n-1)/K\rfloor+1}{n}
\ge\frac1K.
$$
Hence, taking $\varepsilon=1/(2K)$ and
$\delta=\frac12\lambda(J)$, we obtain
$$
\lambda\!\left(\Bigl\{x\in[0,1]:
\frac1n\sum_{j=0}^{n-1}\chi_U(g^j(x))
>\varepsilon\Bigr\}\right)
\ge\lambda(J)>\delta.
$$
This shows that $U$ is an essential set, proving
\ref{thm:g} and completing the proof of the theorem.
\end{proof}

The next theorem contrasts the generic $C^0$ picture obtained here with the finitude scenario proposed by Palis for smooth parameter families~\cite{Palis2}. It shows that generic interval maps possess a canonical observable attractor, but no Palis attractor.

\begin{thm}\label{thm:generic_attr}
There exists a residual family $G\subset C([0,1])$ such that, for every
$g\in G$, the global Milnor, statistical, and physical attractors coincide
with the non-wandering set $\Omega(g)$, which is a Cantor set. Moreover, the
global Milnor attractor is robust but not Lyapunov stable. Every invariant
probability measure of $g$ has a basin of zero Lebesgue measure. Consequently,
$g$ admits no Palis attractor. Finally, the only asymptotically stable sets of
$g$ are finite unions of invariant nondegenerate periodic intervals.
\end{thm}

\begin{proof}
Let $G$ be the set from Theorem \ref{thm:residual}. For any $g\in G$, let $\Amin(g)$, $\Astat(g)$ and $\Lambda(g)$ denote the physical, statistical, and global Milnor attractors of $g$, respectively. We first  claim that $\Amin(g)=\Astat(g)=\Lambda(g)=A(g)=\Omega(g).$ By Lemma \ref{lem:attr}, we have  $\Amin(g)\subseteq \Astat(g)\subseteq\Lambda(g)\subseteq \Omega(g)$ and so by Theorem \ref{thm:residual}\ref{thm:e}, it is enough to show that $A(g)\subseteq \Amin(g)$. 
Suppose instead that $A(g)\setminus\Amin(g)\neq\emptyset$. Choose $z\in A(g)\setminus\Amin(g)$.
By the definition of $\Amin(g)$, there exists
a nonessential open set $U$ containing $z$.
But $U\cap A(g)\neq\emptyset$, so by
Theorem \ref{thm:residual}\ref{thm:g},
$U$ is essential, a contradiction.
Thus $A(g)\subseteq\Amin(g)$,
and the claimed equalities follow.

It is well known that an asymptotically stable set
$A$ admits arbitrarily small trapping neighborhoods
$U$ such that
$g(\overline U)\subset U$ and
$A=\bigcap_{n\ge0}g^n(\overline U)$.
By compactness, we may choose such a neighborhood
to be a finite union of intervals.
Since the number of components cannot increase
under iteration, $A$ is a finite union of
(possibly degenerate) closed intervals.
As $g(A)=A$, its components are periodically
permuted.

None of these components can be degenerate.
Indeed, a singleton component would belong
to a periodic cycle of singleton components.
A sufficiently small trapping neighborhood
of this cycle would contain a nondegenerate
interval of points whose omega-limit sets
are finite, contradicting
Theorem \ref{thm:residual}\ref{thm:b},
since $Z(g)$ has full Lebesgue measure.
Thus every component is nondegenerate.

Since $\Omega(g)=A(g)$ and
$\dim_H A(g)=0$, the set $\Omega(g)$
is totally disconnected.
Moreover, it has no isolated points.
Indeed, if $z_0$ were isolated in $\Omega(g)$,
then $z_0\in\Per(g)$, since
$\Omega(g)=\overline{\Per(g)}$.
By \eqref{eq:periodic-trap}, there would exist
a nondegenerate closed interval $J$ such that
$$
J\cap\Omega(g)=\{z_0\},
\qquad g^K(J)\subset\Int J.
$$
Choose $x\in J\cap Z(g)$.
Then $\omega(x)\cap J\neq\emptyset$,
so $z_0$ would be an isolated point of
the solenoid $\omega(x)$, a contradiction.
Thus $\Omega(g)$ is a Cantor set.

In particular, $\Omega(g)$ is not
asymptotically stable, by the preceding
characterization and $\dim_H\Omega(g)=0$.
Since $B(\Omega(g))=[0,1]$, Lyapunov
stability of $\Omega(g)$ would imply
asymptotic stability.
Hence $\Omega(g)$ is not Lyapunov stable.

We now establish the robustness of the global Milnor attractor at $g$.
Let $\Theta(f) = \bigcup_{x\in X} \omega(x)$ denote the union of all $\omega$--limit sets of a function $f$, and let $\CR(f)$ be its chain recurrent set. 
Every point of an $\omega$--limit set is chain recurrent, so $\Theta(f) \subseteq \CR(f)$, and since $\CR(f)$ is closed, we also have $\overline{\Theta(f)} \subseteq \CR(f)$. 
It follows that $\Lambda(f) \subseteq \overline{\Theta(f)} \subseteq \CR(f)$.
By Theorem \ref{thm:residual}, for any $g\in G$
we have $\Lambda(g)=\Omega(g)$
and also $\Omega(g)=\CR(g)$, since $g$
has the shadowing property (e.g. see the proof of Theorem~3.4.2 in \cite{Aoki}, cf.  \cite{Oprocha}). Now let $g_n \to g$ in the uniform topology on $C([0,1])$. Since the chain recurrent set varies upper semicontinuously under uniform convergence (see Akin \cite[Theorem~7.23]{Akin}), 
$$
\limsup_{n \to \infty} \Lambda(g_n) \subseteq \limsup_{n \to \infty} \CR(g_n) \subseteq \CR(g) = \Lambda(g).
$$ 
For every $x\in Z(g)$, the interval cycles
constructed in Theorem \ref{thm:residual}
satisfy the assumptions of Lemma~\ref{lem:USC}.
Hence
$
\omega(x)\subseteq
\liminf_{n\to\infty}\Lambda(g_n).
$
Since the set on the right is closed and
$$
\Lambda(g)=A(g)=
\overline{\bigcup_{x\in Z(g)}\omega(x)},
$$
we obtain
$
\Lambda(g)\subseteq
\liminf_{n\to\infty}\Lambda(g_n).
$
Therefore $d_H(\Lambda(g_n),\Lambda(g))\to0$, so the global Milnor attractor is robust at $g$.

Suppose that an invariant 
probability measure $\mu$ has a
basin $B_\mu$ of positive Lebesgue measure. Since $Z(g)$ has full measure,
$B_\mu\cap Z(g)$ has positive measure. Fix $x\in B_\mu\cap Z(g)$. For every
$y\in B_\mu\cap Z(g)$, convergence of the empirical measures to $\mu$ gives
$\sigma(y)=\operatorname{supp}\mu$. Since $\omega(y)$ is a solenoid,
Lemma~\ref{lem:solenoid} yields $\sigma(y)=\omega(y)$. Hence
$\omega(y)=\operatorname{supp}\mu=\omega(x)$, and therefore
$$
B_\mu\cap Z(g)\subseteq B(\omega(x)).
$$
This contradicts Theorem~\ref{thm:residual}\ref{thm:c}, which gives
$\lambda(B(\omega(x)))=0$. Thus every invariant 
probability measure
has a basin of zero Lebesgue measure, and consequently $g$ admits no Palis
attractor.
\end{proof}
\begin{proof}[Proof of Theorem A and Theorem B]
Both are immediate consequences of Theorems~\ref{thm:generic_attr} and \ref{thm:residual}.
\end{proof}

\begin{rem}
A less direct but somewhat faster way to prove the coincidence of attractors in Theorem A is to apply
Theorem 1 and Corollary 6 of Catsigeras and Troubetzkoy~\cite{CatTroub2}. The details are left to the reader.
\end{rem}

\begin{rem}
By Theorem~\ref{thm:residual}\ref{thm:a}
and \ref{thm:a'}, the map $g$ is totally singular:
$$
\lambda(Z(g))=1
\quad\text{and}\quad
\lambda(g^{-1}(Z(g)))=0.
$$
Moreover, by Theorem~\ref{thm:generic_attr},
every invariant probability measure of $g$
has a basin of zero Lebesgue measure.
Hence $g$ admits no SRB (physical) invariant
probability measure.
Therefore, $g$ is \emph{weird} in the
terminology of \cite{AA2013}.
\end{rem}

\begin{exmp}
To illustrate the existence of local asymptotically stable
sets in $G$, consider an interval map $f$ with an attracting
periodic point $p$. Then there exists an open neighborhood
$U$ of $\Orb_f(p)$ such that
$f(\overline U)\subset U$.
Now take the set $G$ from Theorem~\ref{thm:residual}.
By density, there is $g\in G$ sufficiently close to $f$
such that $g(\overline U)\subset U$, and hence
$$
A=\bigcap_{n\ge0}g^n(\overline U)
$$
is asymptotically stable.
This shows that, although the global attractor
$\Lambda(g)=\Omega(g)$ is never asymptotically stable,
maps $g\in G$ can still have local asymptotically
stable sets.

In fact, for any $N\ge1$, we can choose $f$ with
$N$ attracting fixed points having pairwise disjoint
trapping neighborhoods $U_1,\dots,U_N$.
By density, there exists $g\in G$ sufficiently
close to $f$ such that
$$
g(\overline{U_i})\subset U_i,
\qquad i=1,\dots,N.
$$
Thus the sets
$$
A_i=\bigcap_{n\ge0}g^n(\overline{U_i}),
\qquad i=1,\dots,N,
$$
are pairwise disjoint and asymptotically stable.
Consequently, there is no uniform upper bound
on the number of disjoint asymptotically stable
sets that a map $g\in G$ can possess.
\end{exmp}

\section{Unobserved, yet important dynamics}\label{sec:maximality-examples}

Theorem~\ref{thm:A} identifies the global Milnor attractor of a typical map with the entire non-wandering set. The aim of this section is to show that this conclusion is substantially stronger than an almost-everywhere description of orbit limits. In particular, solenoidal observable dynamics may coexist with recurrent pieces outside the global Milnor attractor, and the omitted part may carry all of the topological entropy. Specifically, we construct an example for which
$\Lambda(f)\subsetneq\Omega(f)$, while the
observable dynamics remains solenoidal
(see Theorem~\ref{thm:entropy-hidden}).

\begin{exmp}\label{ex:delahaye-hidden-periodic}
Delahaye~\cite{Del} constructed a continuous interval map
with an odometer as a maximal $\omega$-limit set and
countably many repelling periodic points.
The construction was refined in
\cite[Example~2.4]{OL} to obtain uncountably many
odometers: at each stage, every periodic cycle splits
into two smaller cycles with disjoint orbits.

Consider the self-similar affine realization
$G:[0,1]\to[0,1]$ defined by
$$
G(t)=
\begin{cases}
t+\frac23,
&0\le t\le\frac13,\\
1-\frac{25}{9}(t-\frac13),
&\frac13\le t\le\frac23,\\
\frac19G(9t-6),
&\frac23\le t\le\frac79,\\
\frac1{36}+\frac{29}{12}(t-\frac79),
&\frac79\le t\le\frac89,\\
\frac29+\frac19G(9t-8),
&\frac89\le t\le1.
\end{cases}
$$
The first refinement consists of two periodic cycles
whose union is
$$
C=[0,1/9]\cup[2/9,1/3]
\cup[2/3,7/9]\cup[8/9,1].
$$
The return map $G^2$ on each upper component is
affinely conjugate to $G$, with a scaling factor of $1/9$.
Repeating this construction produces uncountably
many solenoids, each conjugate to an odometer.
At level $n\ge1$, the retained set $C_n$ consists
of $4^{n}$ intervals of length $9^{-n}$.
Consequently,
$$
\lambda(C_n)=\left(\frac49\right)^n\to0,
$$
and $C_\infty=\bigcap_{n\ge1}C_n$ is a null
Cantor set formed by the constructed solenoids.

We show that each individual solenoid has a basin
of zero Lebesgue measure. Fix such a solenoid $S$,
and write $B=B_G(S)$. Let $r$ denote the measure
of $B$ in each component of its first periodic
cycle, $x$ its measure in each outer gap, and $z$
its measure in the middle gap.
Since the other periodic cycle is disjoint from
$B$, backward invariance and the affine slopes give
$$
x\le\frac{12}{29}(r+x),\qquad
z\le\frac9{25}(2r+2x+z).
$$
Thus $x\le r$, $z\le4r$, and
$$
\lambda(B)=2r+2x+z\le8r.
$$
By self-similarity,
$r=\frac19\lambda(B_G(S'))$, where $S'$ is the
corresponding solenoid in the normalized return copy.
Iterating, we obtain
$$
\lambda(B_G(S))
\le\left(\frac89\right)^n\to0.
$$

The same affine estimates show that almost every
orbit eventually enters $C$.
Indeed, for the set
$$
W=\bigcap_{n\ge0}G^{-n}([0,1]\setminus C),
$$
the corresponding gap measures satisfy
$$
x\le\frac{12}{29}x,\qquad
z\le\frac9{25}(2x+z),
$$
so $\lambda(W)=0$.
By self-similarity, this holds at every refinement
level. Hence almost every orbit follows a nested
sequence of periodic cycles and has a solenoidal
$\omega$-limit set contained in $C_\infty$.

By Lemma~\ref{lem:solenoid-milnor}, every constructed
solenoid belongs to $\Lambda(G)$.
The almost-everywhere description gives the reverse
inclusion, and therefore
$
\Lambda(G)=C_\infty.
$
On the other hand, the central affine branch has
a repelling fixed point $p=26/51\notin C_\infty$.
Thus
$
\Lambda(G)\subsetneq\Omega(G).
$
The periodic orbits of $G$ have periods which are
powers of two, as follows from the recursive
splitting and the affine gap dynamics.
Hence $h_{\mathrm{top}}(G)=0$.

Finally, we modify the endpoints. Set
$a=1/9$, $b=8/9$, and define
$$
F(t)=
\begin{cases}
\frac{17}{3}t,
&0\le t\le a,\\
a+(b-a)G\left(\frac{t-a}{b-a}\right),
&a\le t\le b,\\
1-\frac{25}{4}(1-t),
&b\le t\le1.
\end{cases}
$$
Then $F(0)=0$, $F(1)=1$, and the restriction
to $[a,b]$ is affinely conjugate to $G$, as illustrated in
Figure~\ref{fig:Cantor-odo}.
The two outer branches are expanding, so every
point outside $[a,b]$, except the endpoints,
eventually enters this interval.
Since these branches are affine with nonzero
slopes, the full-measure description of
$\omega$-limit sets and the zero-measure basins
are preserved. Moreover, $h_{\mathrm{top}}(F)=0$.

Thus, writing $A_F=\Lambda(F)$, we have
$
A_F\subsetneq\Omega(F),
$
although almost every orbit has a solenoidal
$\omega$-limit set with a zero-measure basin.
We shall use $F$ as a building block in
the next result.
\end{exmp}

\begin{figure}[ht]
\centering
\begin{tikzpicture}[x=5cm,y=5cm]
\draw[gray!25,densely dotted] (0,0) -- (1,1);
\begin{scope}[cm={0.8,0,0,0.8,(0.1,0.1)}]
  \draw[gray!25,densely dotted]
    (1/3,0) -- (1/3,1) (2/3,0) -- (2/3,1)
    (0,1/3) -- (1,1/3) (0,2/3) -- (1,2/3);
  \begin{scope}[cm={1/3,0,0,1/3,(2/3,0)}]\FirstGrid{3}\end{scope}
  \Graph{3}
  \draw[gray!25,densely dotted] (0,0) rectangle (1,1);
\end{scope}
\draw[line width=0.1pt] (0,0) -- (0.1,19/30);
\draw[line width=0.1pt] (0.9,3/10) -- (1,1);
\draw[gray!60] (0,0) rectangle (1,1);
\end{tikzpicture}
\caption{A schematic finite-stage approximation of the modified Delahaye construction.}
\label{fig:Cantor-odo}
\end{figure}

\begin{thm}\label{thm:entropy-hidden}
There exists a continuous interval map $S$ whose global
Milnor attractor is the closure of the union of solenoids,
each having a basin of attraction of zero Lebesgue measure.
This attractor carries zero topological entropy and can moreover
be chosen to have Hausdorff dimension zero, while a
disjoint zero-measure Cantor invariant set carries the
full topological entropy of $S$. In particular,
$$
\Lambda(S)\subsetneq\Omega(S).
$$
\end{thm}

\begin{proof}
Let $F:[0,1]\to[0,1]$ be the map from
Example~\ref{ex:delahaye-hidden-periodic}, and let
$A_F=\Lambda(F)$. Recall that $A_F$ is a null Cantor
set formed by the constructed solenoids, each having
a basin of zero Lebesgue measure, that $B(A_F)$ has
full measure, and that $h_{\mathrm{top}}(F)=0$.

Fix $a\in(0,1)$ and define $S:[0,2]\to[0,2]$ by
\[
S(x)=
\begin{cases}
aF(x/a), & x\in[0,a],\\[2mm]
a\left(1-\dfrac{x-a}{1-a}\right), & x\in[a,1],\\[3mm]
4(x-1), & x\in[1,\frac32],\\[3mm]
4(2-x), & x\in[\frac32,2].
\end{cases}
\]
The map is continuous, $S([0,1])\subset[0,a]$,
and $S|_{[0,a]}$ is conjugate to $F$. Put
\[
K=\bigcap_{n\ge0}S^{-n}([1,2]).
\]
The intervals $[5/4,3/2]$ and $[3/2,7/4]$
are two full branches mapped linearly onto $[1,2]$.
Hence $K$ is a Cantor repeller, $S|_K$ is
conjugate to the full shift on two symbols, and
$
h_{\mathrm{top}}(S|_K)=\log2.
$
Since both inverse branches contract by the
factor $1/4$, we have $\lambda(K)=0$.

Every point of $[1,2]\setminus K$ eventually
leaves $[1,2]$ and enters $[0,a]$, possibly
after one iterate in $[a,1]$.
Since each branch of $S$ outside $[0,a]$ is
affine with nonzero slope, preimages of
Lebesgue-null sets have zero Lebesgue measure.
Thus the full-measure basin property of $A_F$
transfers to $aA_F$. Moreover, the restriction
of $S$ to $[0,a]$ is conjugate to $F$, so
$
\Lambda(S)=aA_F.
$

On the other hand,
\[
K\cap\Lambda(S)=\emptyset,
\qquad
\operatorname{dist}(K,\Lambda(S))\ge1-a>0.
\]
Furthermore, every non-wandering point outside
$K$ belongs to $[0,a]$, and hence
\[
\Omega(S)\subset\Omega(S|_{[0,a]})\cup K.
\]
Since $h_{\mathrm{top}}(F)=0$, the entropy
outside $K$ is zero, while
$h_{\mathrm{top}}(S|_K)=\log2$.
Therefore
$
h_{\mathrm{top}}(S)=\log2.
$
Thus the entire topological entropy of $S$
is carried by a zero-measure Cantor repeller
lying at positive distance from the global
Milnor attractor. In particular,
$
\Lambda(S)\subsetneq\Omega(S).
$

For each constructed solenoid of $F$, its
corresponding basin for $S$ consists of its
affine rescaling in $[0,a]$ and countably
many pullbacks along affine first-entry
branches of nonzero slope. Hence its basin
has zero Lebesgue measure. Any solenoid
contained in $K$ also has a zero-measure
basin, since that basin is contained in $K$.

Finally, let $A=aA_F$. Since $A$ is a null
Cantor set, choose a Cantor set $D$ of
Hausdorff dimension zero with the same extrema,
and an increasing homeomorphism $H:A\to D$.
Extend $H$ affinely over the complementary
gaps of $A$, and by the identity outside
$[\min A,\max A]$.
Both $H$ and $H^{-1}$ preserve Lebesgue-null
sets, since $A$ and $D$ are null and the
extensions are affine on countably many gaps.

Consequently, the conjugate map
$H\circ S\circ H^{-1}$ has global Milnor
attractor $D$ and preserves the full-measure
basin and the zero-measure basins of individual
solenoids. The repeller $K$ is unchanged,
and topological entropy is preserved.
Thus the global Milnor attractor can be chosen
to have Hausdorff dimension zero, completing
the proof.
\end{proof}

Theorem~\ref{thm:entropy-hidden} shows that the global Milnor attractor may describe the asymptotic behavior of Lebesgue-almost every initial condition while still failing to capture the topological complexity of the system. In this example, all topological entropy is supported on a disjoint zero-measure Cantor repeller lying outside the global Milnor attractor. The identity
$$
\Lambda(f)=\Omega(f)
$$
established in Theorem~\ref{thm:A} excludes precisely this phenomenon for a residual family of interval maps.

\section{Interval map satisfying $\Amin\subsetneq\Astat\subsetneq \Lambda$}

In this section, we present an example of an interval map whose physical attractor is not equal to the statistical one. It is based on the construction of a carefully chosen full-measure set $X$ of sequences in the full shift, whose image under a suitable map is a subset of full $\lambda$-measure in the interval. Since the final result for the interval depends on the combinatorial and dynamical behavior of the elements of $X$, we first recall some basic definitions and properties from symbolic dynamics.

\subsection{Symbolic dynamics} For the alphabet $\mathcal{A}=\{0,1\}$, let $$\Sigma = \{x = (x_i)_{i\geq 0}: x_i \in \mathcal{A}, i\geq 0\}$$ denote the full shift over $\mathcal{A}$. Let $T:\Sigma \ni (x_i)_{i\geq 0}  \to  (x_{i+1})_{i\geq 0} \in \Sigma$ denote the shift map. 
 Finite sequences of symbols over $\mathcal{A}$ are called \emph{words} or \emph{blocks}.  The \emph{length} of a word is the number of symbols in that word. For any $n\geq 1$, let $\mathcal{L}_n = \{x_0\dots x_{n-1}: x_i \in \mathcal{A}, i=0,\dots,n-1\}$ be the set of all words of length $n$ over the alphabet $\mathcal{A}$ and let $\mathcal{L} = \bigcup_{n\geq 1}\mathcal{L}_n$ be the set of all finite words over $\mathcal{A}$. For any $w \in \mathcal{L}$, the \emph{cylinder set of the word }$w$ is the set $[w]=\{wx: x \in \Sigma\}$, that is, the set of all sequences in $\Sigma$ that begin with $w$. For any $n\geq 1$ and any letter $a \in \mathcal{A}$, we denote by $a^n$ the word of length $n$ consisting of $n$ repetitions of the letter $a$. 

We introduce a natural order on $\mathcal{L}$ as follows:
$$w_1 \prec w_2 \Leftrightarrow (|w_1|<|w_2|) \vee (|w_1|=|w_2|\wedge w_1 <_2 w_2),$$
where $<_2$ denotes the usual order of binary numbers. The first elements in this order are
$$
0 \prec 1 \prec 00 \prec 01\prec 10 \prec 11\prec 000\prec \dots
$$
Throughout the paper, we will use the notation $\mathcal{L} = \{w_n: n\geq 0\}$, assuming that $w_0 = 0$ and $w_{n+1} = \min\{w \in \mathcal{L}: w_n\prec w\}$ for every $n\geq 0$.

For $x \in \Sigma$ and integers $0 \leq k < l$, we use two related notations:
$$
x_{[k,l]} = x_k x_{k+1} \dots x_l, 
\qquad 
x_{[k,l)} = x_k x_{k+1} \dots x_{l-1}.
$$
Both denote finite subsequences of $x$, called \emph{blocks} (or \emph{subwords}) of $x$. Such a subsequence is called a \emph{block of $x$} or a \emph{subword of $x$}. For any $l>0$, a \emph{prefix of length $l$} of the word $x$ (finite or infinite) is $x_{[0,l)}$.  
Denote $\mathbf{0} = 0^{\infty}$ and $\mathbf{1}=1^{\infty}.$ 

\begin{exmp}
We will show that, for the example constructed below,
$$
\Amin = \{0\}
\subsetneq
\Astat = \{0,p\}
\subsetneq
\Lambda = [0,1],
$$
where $p=h(2/3)$ is the nonzero fixed point of the map constructed below.
These inclusions will be established by Lemmas~\ref{lem:Astat}, \ref{lem:Amin}, and~\ref{lem:Lambda}.
\end{exmp}

{\bf Construction}:
We start by defining a family of infinite sequences
$\{x_c\}_{c\in\Sigma}\subset\Sigma$.
Each sequence $x_c$ is obtained by concatenating three types of blocks,
whose lengths depend on the sequence
$k_n=2^{2^n}$, $n\geq0$:
\begin{align*}
A_n &= 0^{k_{2n}}, \qquad n\geq0,\\
B_n^c &=
\begin{cases}
1^{k_{2n+1}-|w_n|} & \text{if } w_n=c_{[0,|w_n|)},\\
0^{k_{2n+1}-|w_n|} & \text{otherwise},
\end{cases}\\
w_n &\in\mathcal{L}.
\end{align*}
In particular, $|A_n|=k_{2n}$ and $|B_n^cw_n|=k_{2n+1}$.
For every $c\in\Sigma$, define
$$
x_c=A_0B_0^cw_0A_1B_1^cw_1A_2B_2^cw_2
A_3B_3^cw_3A_4B_4^cw_4A_5B_5^cw_5
A_6B_6^cw_6\dots
$$

We also introduce two sequences of lengths,
$\{a_n\}_{n\geq0}$ and $\{b_n\}_{n\geq0}$:
\begin{align*}
a_0 &= k_0,\\
a_n &= k_0+\dots+k_{2n}
=\sum_{i=0}^{2n}2^{2^i}, \qquad n\geq1,\\
b_0 &= k_0+k_1,\\
b_n &= k_0+\dots+k_{2n+1}
=\sum_{i=0}^{2n+1}2^{2^i}, \qquad n\geq0.
\end{align*}
For every $n\geq0$ and $c\in\Sigma$, the numbers
$a_n$ and $b_n$ mark the positions immediately after
the blocks $A_n$ and $w_n$, respectively, in $x_c$.

We summarize below several asymptotic relations between the sequences $a_n$, $b_n$, and $k_n$, which will be used repeatedly in the subsequent proofs.

\medskip
\begin{equation}
\label{est:E1}
\frac{a_n}{k_{2n}} \longrightarrow 1,
\qquad
\frac{b_n}{k_{2n+1}} \longrightarrow 1.
\end{equation}
Indeed, since $k_m = 2^{2^m}$ and $k_{m+1}=k_m^2$, each term $k_{m+1}$ 
dominates all preceding ones. Hence 
$a_n = k_{2n} + \sum_{i=0}^{2n-1} k_i \le k_{2n} + 2k_{2n-1}$ and 
$b_n = k_{2n+1} + \sum_{i=0}^{2n} k_i \le k_{2n+1} + 2k_{2n}$, giving 
$1 \le \tfrac{a_n}{k_{2n}} \le 1 + \tfrac{2}{k_{2n-1}}$ and 
$1 \le \tfrac{b_n}{k_{2n+1}} \le 1 + \tfrac{2}{k_{2n}}$, 
which yields the desired limits.

\medskip

For $n \ge 1$, using $k_m = 2^{2^m}$, we have $k_{2n} = (k_{2n-1})^2$, 
so $\tfrac{2k_{2n-1}}{k_{2n}} = \tfrac{2}{k_{2n-1}} \to 0$ as $n \to \infty$.  
Since $a_n = \sum_{i=0}^{2n} k_i \ge k_{2n}$ and 
$\sum_{i=0}^{2n-1} k_i \le 2k_{2n-1}$, we obtain
\begin{equation}\label{est:E2}
\frac{\sum_{i=0}^{n-1} k_{2i+1}}{a_n}
\;\le\;
\frac{2k_{2n-1}}{k_{2n}}
=\frac{2}{k_{2n-1}}
\;\longrightarrow\; 0.
\end{equation}

Denote
$Q = \{x_c : c \in \Sigma\} \subset \Sigma$
and define
$X = \bigcup_{n \ge 0} T^n(Q)$.
Consider the tent map
$$
\varphi(t)=
\begin{cases}
2t, & 0\le t\le 1/2,\\
2(1-t), & 1/2\le t\le 1.
\end{cases}
$$
The full shift $(\Sigma,T)$ is an almost one-to-one
extension of $([0,1],\varphi)$.
To define the coding map $\Phi:\Sigma\to[0,1]$,
consider the inverse branches
$$
r_0(t)=t/2,\qquad r_1(t)=1-t/2,
$$
and, for $s=(s_i)_{i\ge0}\in\Sigma$, put
$$
\Phi(s_0s_1\ldots)
=\lim_{n\to\infty}
r_{s_0}\circ r_{s_1}\circ\cdots\circ r_{s_n}(0).
$$
Then $\varphi\circ\Phi=\Phi\circ T$, and
$$
\Phi(0^\infty)=0,\qquad
\Phi(1^\infty)=2/3.
$$
The map $\Phi$ is injective except for the pairs
$$
w010^\infty,\qquad w110^\infty,
$$
where $w$ ranges over all finite binary words,
including the empty word.
In particular, $\Phi$ is injective on $X$,
since every sequence in $X$ contains infinitely
many prescribed blocks $w_n$ and hence is not
eventually constant.

The image $\Phi(Q)\subset[0,1]$ is a Cantor set, and
$$
\Phi(X)=\bigcup_{n=0}^{\infty}\varphi^n(\Phi(Q))
$$
is a dense countable union of Cantor sets in $[0,1]$.

Let $\kappa$ be the Bernoulli measure on $\Sigma$ with weights
$(1/2,1/2)$, and define $\Psi:\Sigma\to\Sigma$ by $\Psi(c)=x_c$.
Consider the probability measure
$$
\mu=\sum_{n=0}^{\infty}2^{-n-1}
(\Phi\circ T^n\circ\Psi)_*\kappa
$$
on $[0,1]$.
Since $\Phi$ is injective on $X$ and each
$T^n\circ\Psi$ is continuous and injective,
the measure $\mu$ is non-atomic.
Moreover, $\mu$ has full support, because
every $x_c$ has a dense orbit in $\Sigma$.
By construction,
$$
\mu(\Phi(X))=1,
\qquad
\mu(\Phi(Q))\ge\frac12.
$$

Define
$$
h(t)=\mu([0,t]),\qquad t\in[0,1].
$$
Since $\mu$ is non-atomic and has full support,
$h:[0,1]\to[0,1]$ is an increasing homeomorphism
satisfying $h_*\mu=\lambda$.
Consequently, for
$\tilde X=h(\Phi(X))$, we have
$\lambda(\tilde X)=1$.
The constant sequences correspond to the two fixed
points of the tent map:
$$
(h\circ\Phi)(\mathbf{0})=0,
\qquad
(h\circ\Phi)(\mathbf{1})=h(2/3)=:p.
$$

Next, define
$
f=h\circ\varphi\circ h^{-1}
$
and the probability measure
$$
\nu=\sum_{n=0}^{\infty}2^{-n-1}
(T^n\circ\Psi)_*\kappa
$$
on $\Sigma$.
Since $\Phi$ is injective on $X$, this is precisely
the pullback of $\mu$ under $\Phi|_X$.
In particular, $\nu$ is non-atomic,
$\nu(X)=1$, and
$
(h\circ\Phi)_*\nu=\lambda.
$
Thus $\nu$ is concentrated on $X$, and the
symbolic and interval systems are related by
the measure-preserving semiconjugacy $h\circ\Phi$.

The relations among the maps defined above are summarized in the following commutative diagram:

\begin{equation*}
  \begin{tikzcd}
{(\Sigma, \nu)} \arrow[rr, "\Phi"] \arrow[dd, "T"] &  & {([0,1], \mu)} \arrow[rr, "h"] \arrow[dd, "\varphi"] &  & {([0,1], \lambda)} \arrow[dd, "f"] \\
                                                      &  &                                             &  &                                  \\
{(\Sigma, \nu)} \arrow[rr, "\Phi"]                      &  & {([0,1]  ,\mu}) \arrow[rr, "h"]               &  & {([0,1], \lambda)}                
\end{tikzcd}  
\end{equation*}

\begin{lem}\label{lem:Astat}
For the interval map $f$ defined above, we have
$\Astat(f) = \{0,p\}$, where $p=h(2/3)$.
\end{lem}

\begin{proof}
Fix $c \in \Sigma$. For any $K \ge 1$, let $U_0 = [0^K]$ and $U_1 = [1^K]$.
For each $n > 0$, consider the block $A_n = 0^{k_{2n}}$ in $x_c$.
Since $A_n$ occupies the coordinates $b_{n-1},\ldots,a_n-1$, all iterates
$T^j(x_c)$ for $b_{n-1}\le j<a_n-K$
lie in $U_0$. Hence, by \eqref{est:E1},
$$
\frac{1}{a_n} \sum_{j=0}^{a_n-1} \chi_{U_0}(T^j(x_c))
\ge
\frac{k_{2n}-K}{a_n}
\longrightarrow 1,
$$
which means that $\mathbf{0}\in\sigma_T(x_c)$.

Since every finite prefix of $c$ occurs among the words $w_n$,
there are infinitely many $n$ for which
$B_n^c = 1^{k_{2n+1} - |w_n|}$.
For those $n$, the orbit segment
$\{T^j(x_c): a_n \le j < a_n + k_{2n+1} - |w_n| - K\}$
lies in $U_1$. Thus, by \eqref{est:E1},
$$
\frac{1}{b_n} \sum_{j=0}^{b_n-1} \chi_{U_1}(T^j(x_c))
\ge \frac{k_{2n+1} - |w_n| - K}{b_n} \longrightarrow 1,
$$
and $\mathbf{1} \in \sigma_T(x_c)$.

If $v$ is any non-constant finite word, visits of the orbit to the cylinder $[v]$ occur only inside the short blocks $w_n$ or within a $|v|$-neighborhood of the boundaries between consecutive blocks. For all sufficiently large $n$ and every $N$ with
$b_{n-1}\le N\le b_n$, we have
\[
0\le
\frac{1}{N}\sum_{j=0}^{N-1}\chi_{[v]}(T^j(x_c))
\le
\frac{\sum_{i=0}^{n}(|w_i|+3|v|)}{b_{n-1}}
\longrightarrow 0.
\]
Indeed, the numerator grows at most polynomially in $n$, whereas
$b_{n-1}$ grows at least exponentially. Since this bound applies to
every sufficiently large averaging time $N$, we conclude that
\[
\lim_{N\to\infty}\frac{1}{N}
\sum_{j=0}^{N-1}\chi_{[v]}(T^j(x_c))=0.
\]
Therefore $\sigma_T(x_c)=\{\mathbf{0},\mathbf{1}\}$ for all $c\in\Sigma$.
As $X=\bigcup_{n\ge0}T^n(Q)$ and $\nu(X)=1$, it follows that
$\sigma_T(x)=\{\mathbf{0},\mathbf{1}\}$ for $\nu$-almost every $x\in\Sigma$.

The map $h\circ\Phi$ is one-to-one on $X$ and satisfies
$$
(h\circ\Phi)(\mathbf{0})=0,
\qquad
(h\circ\Phi)(\mathbf{1})=h(2/3)=p.
$$
Since $\lambda(\tilde X)=\lambda(h(\Phi(X)))=1$, for $\lambda$-almost every
$\tilde x\in[0,1]$ there exists a unique $x\in X$ such that
$\tilde x=(h\circ\Phi)(x)$.
By continuity and the semiconjugacy,
$$
\sigma_f(\tilde x)
=(h\circ\Phi)(\sigma_T(x))
=\{0,p\}
$$
for $\lambda$-almost every $\tilde x$.

By Theorem~\ref{def_stat}, $\Astat(f)$ is the smallest closed set containing
the $\sigma$-limit sets of almost all orbits. Consequently,
$
\Astat(f)=\{0,p\}.
$
\end{proof}

\begin{lem}\label{lem:no_nbhd_of_1_is_essential}
The cylinder $[1]$ is not essential for the system $(\Sigma,T,\nu)$.
\end{lem}

\begin{proof}
Recall that $\Psi(c)=x_c$ and
$$
\nu=\sum_{m=0}^{\infty}2^{-m-1}(T^m\circ\Psi)_*\kappa,
$$
where $\kappa$ is the Bernoulli measure on $\Sigma$.
For $N\ge1$, put
$$
F_N(x)=\frac1N\sum_{j=0}^{N-1}\chi_{[1]}(T^j(x)).
$$

Fix $\varepsilon>0$ and let $b_{n-1}\le N\le b_n$.
If $c\notin[w_{n-1}]\cup[w_n]$, then both
$B_{n-1}^c$ and $B_n^c$ consist of zeros.
For $N\le a_n$, the only possible contributions
to $F_N(x_c)$ come from earlier blocks of ones
and the short words $w_i$. Hence
$$
F_N(x_c)\le
\frac{\sum_{i=0}^{n-2}k_{2i+1}
+\sum_{i=0}^n|w_i|}{b_{n-1}}
\longrightarrow0.
$$
Similarly, for $a_n\le N\le b_n$,
$$
F_N(x_c)\le
\frac{\sum_{i=0}^{n-1}k_{2i+1}
+\sum_{i=0}^n|w_i|}{a_n}
\longrightarrow0
$$
by \eqref{est:E2}. Here we use that
$\sum_{i=0}^n|w_i|$ grows at most polynomially,
whereas $k_n$ grows doubly exponentially.
Consequently, for all sufficiently large $N$,
$$
F_N(x_c)>\varepsilon/2
\quad\Longrightarrow\quad
c\in[w_{n-1}]\cup[w_n].
$$

Now fix $\delta>0$ and choose $M$ so large that
$$
\sum_{m>M}2^{-m-1}<\delta.
$$
For every $m\le M$,
$$
|F_N(T^m(x_c))-F_N(x_c)|\le\frac{m}{N}.
$$
Thus, for sufficiently large $N$,
$$
F_N(T^m(x_c))>\varepsilon
\quad\Longrightarrow\quad
c\in[w_{n-1}]\cup[w_n].
$$
Using the definition of $\nu$, we obtain
\begin{align*}
\nu\{x\in\Sigma:F_N(x)>\varepsilon\}
&\le \delta+
\sum_{m=0}^{M}2^{-m-1}
\kappa([w_{n-1}]\cup[w_n])\\
&\le \delta+
2^{-|w_{n-1}|}+2^{-|w_n|}.
\end{align*}
Since $|w_n|\to\infty$ and $\delta>0$ is arbitrary,
we conclude that
$
F_N\longrightarrow0
\quad\text{in $\nu$-measure}.
$
Therefore $[1]$ is not essential.
\end{proof}

\begin{lem}\label{lem:Amin}
$\Amin(f) = \{0\}$.
\end{lem}

\begin{proof}
By definition, $\Amin(f) \subseteq \Astat(f)$, and from
Lemma~\ref{lem:Astat} we know that $\Astat(f) = \{0,p\}$,
where $p=h(2/3)$. To prove the claim, it suffices to show
that $p \notin \Amin(f)$.

Since $\Phi([1])=[1/2,1]$ and $p=h(2/3)$,
we can choose an open neighborhood $U$ of $p$ such that
$
U\subset h((1/2,1)).
$
Then, for every $x\in\Sigma$,
$$
\chi_U((h\circ\Phi)(x))\le\chi_{[1]}(x).
$$
By Lemma~\ref{lem:no_nbhd_of_1_is_essential},
the time averages of $\chi_{[1]}$ converge to zero
in $\nu$-measure. Since
$(h\circ\Phi)_*\nu=\lambda$, the semiconjugacy implies
that the time averages of $\chi_U$ converge to zero
in Lebesgue measure. Thus $U$ is not essential,
and consequently $p\notin\Amin(f)$.

As $\Amin(f)$ contains the support of any weak$^*$
accumulation point of
$\frac1n\sum_{i=0}^{n-1}f^i_*\lambda$, it is nonempty.
Finally, $\Amin(f)\subseteq\Astat(f)=\{0,p\}$,
so we conclude that $\Amin(f)=\{0\}$.
\end{proof}

\begin{lem}\label{lem:Lambda}
$\Lambda(f) = [0,1]$.
\end{lem}

\begin{proof}
By construction, every finite word occurs infinitely
often in each $x_c$, so $\omega_T(x_c)=\Sigma$
for every $c\in\Sigma$.
The same holds for every point of
$X=\bigcup_{n\ge0}T^n(Q)$.
Since $\Phi$ semiconjugates $T$ to the tent map
$\varphi$ and $h$ is a homeomorphism of $[0,1]$,
we obtain
$$
\omega_f((h\circ\Phi)(x))=[0,1]
\qquad\text{for every }x\in X.
$$
As $\lambda(h(\Phi(X)))=1$, almost every point
of $[0,1]$ has the whole interval as its
$\omega$-limit set. Hence $\Lambda(f)=[0,1]$.
\end{proof}
\section*{Acknowledgements}
Magdalena Fory\'s-Krawiec and Piotr Oprocha were partially supported by the National Science Centre, Poland (NCN), grant no.\ 2024/53/B/ST1/00092. Jana Hant\'akov\'a was supported by institutional support for the development of research organizations in Czech Republic (RVO) for IČ 47813059.

For the purpose of Open Access, the authors have applied a CC BY public copyright licence to any Author Accepted Manuscript (AAM) version arising from this submission.

\section*{Statements and declarations}
On behalf of all authors, the corresponding author states that there is no
conflict of interest and that data sharing is not applicable to this article as no datasets were generated or analyzed during the current study.

Recently, the mathematical capabilities of AI models have expanded, and reviewers have started to use them in the evaluation of papers. Consequently, their use in manuscript preparation has become unavoidable.
While the results in the paper were obtained some time ago, during the revision of this manuscript, the authors used ChatGPT to improve the clarity of proofs and the grammar and sentence structure in selected sections of the text. The authors reviewed and revised all suggestions, verified technical accuracy, and take full responsibility for the final content.

\end{document}